\documentclass[reqno]{amsart}
\usepackage{amsmath,amsfonts,amssymb,xcolor,enumitem,upgreek,graphicx,comment,caption}

\usepackage[english]{babel}
\usepackage[colorlinks=true,linkcolor=red!60!black,citecolor=green!60!black]{hyperref} %para PDFLaTeX
\usepackage{algorithm}
\usepackage{algpseudocode}

\usepackage[normalem]{ulem}

\usepackage{comment}

\newtheorem{teor}{Theorem}[section]

\newtheorem{prop}[teor]{Proposition}
\newtheorem{coro}[teor]{Corollary}
\theoremstyle{definition}

\newtheorem{remark}[teor]{Remark}

\newcommand{\R}{\mathbb R}
\newcommand{\N}{\mathbb N}
\newcommand{\T}{\mathbb T}
\newcommand{\Z}{\mathbb Z}
\newcommand{\ep}{\varepsilon}
\newcommand{\lb}{\lambda}

\newcommand{\w}{\omega}

\newcommand{\mB}{\mathcal{B}}
\newcommand{\mC}{\mathcal{C}}
\newcommand{\mD}{\mathcal{D}}
\newcommand{\mF}{\mathcal{F}}
\newcommand{\mI}{\mathcal{I}}

\newcommand{\mK}{\mathcal{K}}

\newcommand{\mN}{\mathcal{N}}

\newcommand{\mT}{\mathcal{T}}
\newcommand{\mU}{\mathcal{U}}
\newcommand{\mV}{\mathcal{V}}
\newcommand{\mW}{\mathcal{W}}

\newcommand{\dist}{\mathrm{dist}}

\newcommand{\bx}{x}

\numberwithin{equation}{section}

\begin{document}
\title[Asymptotically Autonomous Maps]{Asymptotically Autonomous Maps: Heteroclinic connections, Splitting and Phase-induced tipping}
\author[J. Due\~{n}as]{Jes\'{u}s Due\~{n}as}
\address[J. Due\~{n}as]{Departamento de Matem\'{a}tica Aplicada, Universidad de Va\-lladolid. Instituto de Investigaci\'{o}n en Matem\'{a}ticas, Universidad de Valladolid.}
\email[J.~Due\~{n}as]{jesus.duenas@uva.es}
\author[A. Vieiro]{Arturo Vieiro}
\address[A. Vieiro]{Departament de Matemàtiques i Informàtica, Universitat de Barcelona.}
\email[A. Vieiro]{vieiro@maia.ub.es}
\thanks{J.~Due\~{n}as is supported by the Spanish Grant No. PID2024-156691NB-I00 funded by MICIU/AEI/10.13039/501100011033 and FEDER Funds, and by Department of Education of the Junta de Castilla y León and FEDER Funds under project CLU-2025-1-02-IMUVA.
A.~Vieiro is supported by the Spanish Grant No. PID2021-125535NB-I00 funded by MICIU/AEI/10.13039/501100011033 and by ERDF/EU}
\date{}
\begin{abstract}
We consider discrete-time dynamical systems generated by the iteration of asymptotically autonomous maps, whose asymptotic behavior may be either conservative or dissipative. These systems are of particular interest in the modeling of complex phenomena with time-dependent parameter variation.
With an appropriate time compactification, the past and future infinity become
normally hyperbolic invariant hyperplanes, and thus the plausible asymptotic
states of the system can be determined by considering the appropriate fibers of
the stable and unstable foliations or suitable sections of the asymptotic
manifolds to their compact invariant subsets. In this context, we explore how
tipping induced by a change of phase--namely,
how a phase shift affects the evolution of the system--is encoded in  
the unstable manifolds of compact invariant subsets of the past equation.
We further connect the tipping phenomenon to the splitting of invariant
manifolds and identify certain system reversibilities able to prevent tipping.
The theoretical framework we use clarifies the relationships among the various
notions of tipping induced by phase-change introduced in the literature.
\end{abstract}

\renewcommand{\subjclassname}{\textup{2020} Mathematics Subject Classification}

\maketitle
\section{Introduction}
Nonautonomous dynamical systems provide a natural framework for modeling processes whose governing laws vary in time \cite{KloedenPotzsche2013}.
An important subclass is formed by asymptotically autonomous systems \cite{Markus1956}, for which the time-dependent dynamics converge, in the distant past and future, to autonomous limiting systems.
These systems are particularly relevant to the theory of critical transitions \cite{apw2017}, where the emphasis is placed on the behavior of local pullback attractors.
Although the asymptotic limit problems are autonomous, the transition between them may generate dynamical phenomena that cannot be inferred from either limit system separately.
Consequently, their analysis may require genuinely nonautonomous notions of attraction, repulsion, and bifurcation \cite{AnagnostopoulouPotzscheRasmussen2023,rasmussen2007,rasmussen2008}.

A powerful geometric approach to such problems under certain additional conditions was systematically developed by \cite{wieczorek2021}, with the underlying compactification idea already present in \cite{AlexanderGardnerJones1990} and a related formulation appearing in \cite{AshwinWieczorekVitoloCox2012}.
For asymptotically autonomous differential equations, time is compactified and the original nonautonomous problem is embedded into an autonomous dynamical system on an extended phase space.
The past and future limit systems then appear as invariant boundary subsystems of the compactified problem.
We will work under assumptions ensuring that the compactification is sufficiently smooth and that the invariant hyperplanes corresponding to the past and future limits are uniformly hyperbolic.
In this setting, asymptotic solutions of the original equation can be suitably studied by means of invariant set and invariant manifold methods in the extended autonomous system.

We use this procedure to study nonautonomous recurrences
\[
x_{n+1}=g\left(\Lambda(n),\,x_n\right)\,,\quad x_n\in\R^d\,,
\]
where $\Lambda\colon\R\to\R$ is asymptotically constant, not necessarily monotone or one-to-one,
such that the right-hand side of the recurrence converges, as $n\to-\infty$ and
$n\to+\infty$, to autonomous past and future limit maps $g(\lb_-,x)$ and $g(\lb_+,x)$,
where $\lambda_\pm = \lim_{n \rightarrow \pm \infty} \Lambda(n) \in \R$.
To compactify the time variation, we first regard the discrete index $n$ as the sampling, at integer times, of a continuous variable $t\in\R$, whose evolution is given by the translation $t\mapsto t+1$.
A diffeomorphism
$\Theta\colon\R\to(0,1)$ is used to set $\tau=\Theta(t)$, and define 
$A(\tau)=\Theta\left(\Theta^{-1}(\tau)+1\right)$
and $\tilde\Lambda(\tau)=\Lambda\left(\Theta^{-1}(\tau)\right)$ for $\tau\in(0,1)$.
Both $A$ and $\tilde\Lambda$ can be continuously extended to $[0,1]$.
The original recurrence can therefore be written as
\[
(\tau_{n+1},x_{n+1})=F(\tau_n,x_n)\,,
\]
where the triangular map $F\colon[0,1]\times\R^d\to[0,1]\times\R^d$ is given by
\[
F(\tau,x)=\big(A(\tau),\,g(\tilde\Lambda(\tau),x)\big)\,.
\]
Thus, questions concerning solutions that are asymptotic to 
orbits of the limit maps can be reformulated as questions concerning invariant  manifolds
of the map $F$.

Under the well-established conditions briefly summarized in Section~\ref{sec:2framework}, the time compactified system $F$ realizes the asymptotic dynamics of the original recurrence on the normally  hyperbolic invariant hyperplanes $\{\tau=0\}$ and $\{\tau=1\}$.
Consequently, the dynamics of $F$ can be related with the dynamics on these hyperplanes via
the stable and unstable foliation theorems for normally
hyperbolic invariant manifolds \cite{eldering2013,fenichel1974,hirsch1977}.
In particular, we discuss how these results characterize the unique unstable invariant manifold associated with a compact invariant set of the past limit map and, dually, the unique stable invariant manifold associated with a compact invariant set of the future limit map.
Under certain conditions, this characterization applies to past compact invariant sets that are not the $\upalpha$-limit set of any point outside the set, and to future compact invariant sets that are not the $\upomega$-limit set of any exterior point.
The approach thus identifies distinguished invariant manifolds associated with asymptotic invariant objects of the past system, even in situations when no local pullback attractors emanate from those objects.
It may therefore be viewed as a modest extension of the usual pullback-attractor viewpoint to invariant objects with more general stability properties.
 Within the usual pullback-attractor viewpoint, tipping could be identified with the loss of stability of a local pullback attractor; in the present setting, we instead identify it with a change in the $\upomega$-limit sets of trajectories on the relevant unstable manifold.
Global bounded solutions connecting prescribed past and future asymptotic dynamics can then be described through intersections of the corresponding unstable and stable manifolds in the extended phase space~\cite{wieczorek2023}.

 Asymptotically autonomous systems may exhibit different types of tipping behavior~\cite{AshwinWieczorekVitoloCox2012}: rate-induced, phase-induced, bifurcation-induced, noise-induced...
 For asymptotically autonomous maps, rate-induced tipping was studied in~\cite{kiers2020}.
  In this work, we use the previous geometric viewpoint and focus specifically on how tipping outcomes depend on phase.
Phase tipping was introduced to describe situations in which a time-dependent perturbation causes a transition only from a proper subset of initial phases on an attracting oscillatory state of the past system~\cite{alk2021}.
A related notion of phase-induced tipping has been considered, where a parameter translates the function $\Lambda$ and thus plays the role of a phase; varying this phase parameter may lead to a transition in the $\upomega$-limit set of distinguished trajectories~\cite{lno1,dno3,LongoNunezObaya2021,longonunezobaya}.
In our setting, the geometric description in terms of the unstable manifold associated with the past invariant set and its interaction with the stable manifolds and basin boundaries of the future dynamics helps to clarify the relationship between these two notions of tipping due to changes in phase in the asymptotically autonomous framework.
The corresponding discussion is included in Section~\ref{subsec:phaseinducedtipping} where we further exploit reversibility \cite{lamb1998} to provide information about tipping or to preclude it. Reversible dynamical systems possess a time-reversing symmetry that relates forward and backward dynamics and, under suitable assumptions, ensures the existence of heteroclinic connections between known past and future invariant objects.
Some continuous-time analogues of the biasymptotic constructions considered here  were discussed in
\cite{scarcella2024,scarcella2024-2} 
where, following the ideas in \cite{canadell2015},
the existence of nonautonomous invariant KAM tori in the asymptotically autonomous Hamiltonian setting was considered.

We complement the theoretical results with a collection of analytical and numerical examples illustrating the notions introduced above and the invariant-manifold mechanisms underlying them.
The relevant invariant manifolds are computed and represented using the parameterization method \cite{mamotreto,harodelallave2006}.
In particular, 
a simple example in Section~\ref{sec:41numericalevidenceandresults} 
shows how the splitting of a separatrix can be interpreted as a phase-induced tipping phenomenon: variations of the relevant phase modify the intersection pattern of the associated invariant manifolds and thereby determine  the possible $\upomega$-limits of the trajectories.
This interpretation is related to the broader study of separatrix splitting under 
time-periodic perturbations \cite{fonsimvie19,murillo}.
In our setting, the invariant manifold splitting is
reformulated
in terms of phase-dependent changes in the $\upomega$-limit set of distinguished trajectories.
Further examples in Section~\ref{sec:42surgicalintersections} distinguish the different notions of phase-dependent tipping considered in the literature, showing that they correspond to
related but not generally equivalent mechanisms.

\section{Framework and invariant manifolds} \label{sec:2framework}
In this section, we present a framework to study some invariant structures for
the dynamical system induced by the iteration of a time-dependent map that is
asymptotically autonomous whenever time-compactification is possible.
More specifically, starting from given  compact invariant sets
of the limiting systems, typically invariant tori, we investigate the existence
and properties of asymptotic, as time tends to $\infty$ or $-\infty$, invariant
manifolds to them.

In the spirit of \cite{apw2017,kiers2020}, for $r\geq1$, we consider a
$\lb$-parametric family of recurrence equations given by a $C^r$ map
$g\colon\mI\times\R^d\to\R^d$, where $\mI$ is a real interval.
For each $\lb\in\mI$, we consider an autonomous recurrence relation
\begin{equation}
    \label{eq:2frozen_equation}
    x_{n+1}=g(\lb,x_n)\,,\quad x_n\in\R^d\,,
\end{equation}
where we assume $g(\lb,\cdot)$ to be $C^r$ invertible for every $\lb\in\mI$.
To introduce a variation on the parameter $\lb$ as $n$ varies, we replace $\lb$ with a $C^r$ function $\Lambda:\R\to\mI$ such that there exist two parameter values $\lb_-,\lb_+\in\mI$ satisfying
\[
\lim_{t\to\pm\infty}\Lambda(t)=\lb_\pm\,,\quad\lim_{t\to\pm\infty}\Lambda'(t)=0\,.
\]
Replacing the parameter $\lb$ with the map $\Lambda$ in \eqref{eq:2frozen_equation}, we obtain the asymptotically constant recurrence equation
\begin{equation}
    \label{eq:2nonautonomous_Lambda}
    x_{n+1}=g(\Lambda(n),\,x_n)\,,\quad x_n\in\R^d\,,
\end{equation}
whose asymptotic limits as $n\to\pm\infty$, which we call the past and future equations of \eqref{eq:2nonautonomous_Lambda}, are given by
\begin{equation}\label{eq:2autonomous_limits}
x_{n+1}=g(\lb_-,x_n)\quad\text{and}\quad x_{n+1}=g(\lb_+,x_n)\,.
\end{equation}
To describe the dynamics of \eqref{eq:2nonautonomous_Lambda}, we introduce the triangular map
\begin{equation}
    \label{eq:2trivial_skewproduct}
h:\R\times\R^d\to \R\times\R^d\,,
\qquad
h(t,x)=\bigl(t+1,\,g(\Lambda(t),x)\bigr)\,,
\end{equation}
whose invertibility follows from that of $g(\lb,\cdot)$ for all $\lb\in\mI$, and
whose inverse can then be shown to be of class $C^r$.
In this formulation, the original nonautonomous recurrence
\eqref{eq:2nonautonomous_Lambda} is viewed as the iteration of the map $h$ on
the invariant set $\Z\times\R^d$ of the extended phase space $\R\times\R^d$:
\[
(n,\,x_{n})=h^{n}(0,x_0)\,,\quad n\in\Z\,.
\]
A key point in this paper is that the extended phase space $\R\times\R^d$ of the map $h$ contains, in addition to the invariant set $\mathbb{Z}\times\mathbb{R}^d$, 
on which equation \eqref{eq:2nonautonomous_Lambda} is recovered as explained, 
many other invariant sets of the form $(\mathbb{Z}+c)\times\mathbb{R}^d$ for $c\in\R$, on which recurrence equations of the form
\begin{equation}\label{eq:2phaseinducedtipping}
x_{n+1}=g(\Lambda(n+c),x_n)\,,
\end{equation}
govern the evolution.
The study of the changes in  the dynamics of \eqref{eq:2phaseinducedtipping} as
the parameter $c$ varies is a problem known as \emph{phase-induced tipping},
which has appeared in the recent literature on critical transitions as a
possible mechanism capable of producing abrupt changes in the dynamics of
certain systems under small perturbations \cite{dno3,dno4,longonunezobaya}.
The study of the dynamical system generated by $h$  
allows the simultaneous analysis of all the equations
\eqref{eq:2phaseinducedtipping} arising from a phase shift in the function
$\Lambda$.
The analysis for $c\in\R$ can in fact be reduced to the case $c\in[0,1)$, since
the equation obtained under the change $c\mapsto c+1$ is the same, up to a shift
in the indexing of $n$.
In Section~\ref{subsec:phaseinducedtipping}, we present some relations of this
concept with those of \emph{phase tipping} \cite{alk2018,alk2023,alk2021}.

Following~\cite{wieczorek2021}, we study the dynamics of
\eqref{eq:2nonautonomous_Lambda}
through a compactification of the first variable of the phase space $\R\times\R^d$ of the map $h$. 
As in later sections, we will use the parameterization method (see, for example,
\cite{mamotreto}) to obtain accurate approximations of the relevant invariant
manifolds of compact subsets of such system, we are then interested in providing
a compactification that is at least of class $C^r$.

To this end,  given a diffeomorphism $\Theta\colon\R\to(0,1)$, we define
\begin{equation}\label{eq:torecall2definitionLambdaHat}
\begin{split}
\hat\Lambda\colon [0,1]&\to\R\\
\tau&\mapsto\left\{\begin{array}{ll}
\lb_-\,,&\text{if }\tau=0\,,\\
\Lambda\circ\Theta^{-1}(\tau)\,,\,\;&\text{if }\tau\in(0,1)\,,\\
\lb_+\,,&\text{if }\tau=1\,.\\
\end{array}\right.
\end{split}
\end{equation}
and also
\begin{equation}\label{eq:torecall3definitionA}
\begin{split}
A\colon [0,1]&\to[0,1]\\
\tau&\mapsto\left\{\begin{array}{ll}
0\,,&\text{if }\tau=0\,,\\
\Theta(\Theta^{-1}(\tau)+1)\,,&\text{if }\tau\in(0,1)\,,\\
1\,,&\text{if }\tau=1\,.\\
\end{array}\right.
\end{split}
\end{equation}

\begin{prop}
\label{prop:2smoothness}
Let $r\geq 1$, let $\Theta$ be a $C^r$ diffeomorphism, and let
\[
L_\Theta := \frac1{\Theta'(t)}\frac{d}{dt}\,.
\]
The following assertions hold:
\begin{itemize}
    \item[{\rm(i)}]
    If for every  $1\leq k\leq r$ , the limits
    \[
    a_k^-:=\lim_{t\to-\infty}L_\Theta^k\,\Lambda(t)\,,
    \qquad
    a_k^+:=\lim_{t\to+\infty}L_\Theta^k\,\Lambda(t)
    \]
    exist and are finite, then $\hat\Lambda$ is of class $C^r$.
    Moreover, whenever these conditions hold, one has $\hat\Lambda(0)=\lb_-$, $\hat\Lambda(1)=\lb_+$ and
    \[
    \hat\Lambda^{(k)}(0)=a_k^-,
    \qquad
    \hat\Lambda^{(k)}(1)=a_k^+
    \qquad \text{for all } 1\le k\le r.
    \]

    \item[{\rm(ii)}] 
    If for every  $1\leq k\leq r$, the limits
    \[
    b_k^-:=\lim_{t\to-\infty}L_\Theta^k\,\big(\Theta(\cdot+1)\big)(t)\,,
    \qquad
    b_k^+:=\lim_{t\to+\infty}L_\Theta^k\,\big(\Theta(\cdot+1)\big)(t)
    \]
    exist and are finite, then $A$ is of class $C^r$.
\end{itemize}
\end{prop}

\begin{proof}
(i) For $\tau\in(0,1)$, the chain rule gives
\[
\frac{d}{d\tau}\bigl(\Lambda\circ\Theta^{-1}(\tau)\bigr)=\frac{1}{\Theta'(\Theta^{-1}(\tau))}\,\Lambda'(\Theta^{-1}(\tau)
=(L_\Theta\,\Lambda)(\Theta^{-1}(\tau))\,,
\]
and, by induction, 
\[
\frac{d^k}{d\tau^k}\bigl(\Lambda\circ\Theta^{-1}(\tau)\bigr)
=(L_\Theta^k\,\Lambda)(\Theta^{-1}(\tau))
\qquad \text{for all }0\le k\le r.
\]

For each $0\le k\le r$, we define
\[
F_k(\tau):=
\begin{cases}
a_k^-\,, & \tau=0\,,\\[1mm]
(L_\Theta^k\,\Lambda)(\Theta^{-1}(\tau))\,, & \tau\in(0,1)\,,\\[1mm]
a_k^+\,, & \tau=1\,.
\end{cases}
\]
By the assumed existence of the limits $a_k^\pm$, each $F_k$ is continuous on
$[0,1]$, and
\[
F_k'(\tau)=F_{k+1}(\tau)\,,\quad\text{for }\tau\in(0,1),\quad 0\le k\le r-1.
\]
Continuity allows to conclude that $F_k'=F_{k+1}$ on $[0,1]$ for every
$0\le k\le r-1$.
Iterating, it follows that $F_0=\hat\Lambda$ is of class $C^r$, and that
\[
\hat\Lambda^{(k)}=F_k \qquad \text{for all }0\le k\le r.
\]
In particular, $\hat\Lambda(0)=\lambda_-$, $\hat\Lambda(1)=\lambda_+$ and, for $1 \leq k \leq r$, one has
\[
\hat\Lambda^{(k)}(0)=F_k(0)=a_k^-,
\qquad
\hat\Lambda^{(k)}(1)=F_k(1)=a_k^+,
\]
as we wanted to see.
\medskip

(ii) The proof is analogous to that of {\rm(i)}.
\end{proof}

\medskip
\medskip

The specific properties of the compactified system depend on the theoretical framework here described and, hence, on the particular functions $\Theta$ and $\Lambda$ considered.
The function $\Theta$ is defined in the course of the analysis while, in contrast,
the function $\Lambda$ is considered to be given by the physical problem under consideration.

\medskip

{\bf The choice of $\Theta$.}
Reference \cite[Section 4]{wieczorek2021} presents several types of $C^1$-compactifications, some exhibiting exponential  
decay 
while others show only inverse-polynomial decay as $\tau \rightarrow \pm 1$.
From now on, and throughout all the examples considered in this paper, 
we use the compactification defined via the diffeomorphism 
$\Theta\colon\R\to(0,1)$ 
given by
\begin{equation}\label{eq:2Thetadefinition}
\Theta(t)=\frac1{1+e^{-\alpha\,t}}\,,\qquad\Theta^{-1}(\tau)=\frac1{\alpha}\log\left(\frac{\tau}{1-\tau}\right)\, ,
\end{equation}
for certain $\alpha>0$,
so that the compatified system is expected to show exponential convergence of solutions as $\tau \rightarrow \pm 1$.
Then, for this choice of $\Theta$, one has
\begin{equation}\label{eq:2Adefinition}
A(\tau)=\frac{e^{\alpha}\,\tau}{1+\tau\,(e^\alpha-1)}\,,
\end{equation}
which is in fact real-analytic.
In particular,
\[
A'(\tau)=\frac{e^\alpha}{\big(1+(e^\alpha-1)\,\tau\big)^2}\,,\qquad A'(0)=\frac1{A'(1)}=e^\alpha>1\,.
\]

This choice is particularly useful because it satisfies the symmetry property
\[
A(\tau)+A^{-1}(1-\tau)=1
\]
for all $\tau\in[0,1]$.

\medskip

{\bf On the role of $\Lambda$.}
As said, the function $\Lambda$ is regarded as given by the physical problem in question. However, the
fulfillment of the hypotheses of Proposition \ref{prop:2smoothness}(i) depends on $\Lambda$ and may restrict the possible values of $\alpha$. Note that
\[
\Theta_\alpha\big(\Theta^{-1}_{\alpha/k}(\tau)\big)=\frac{\tau^k}{\tau^k+(1-\tau)^k}\quad\text{for }k\in\N\,.
\]
Since the composition $\Theta_\alpha\circ\Theta^{-1}_{\alpha/k}$ is real‑analytic, 
we deduce that if the compactification $\hat\Lambda_\alpha$ of a map $\Lambda$ through $\Theta_\alpha$ is $C^r$, then its compactification $\hat\Lambda_{\alpha/k}$ through $\Theta_{\alpha/k}$ is also $C^r$,
and that if $\hat\Lambda_\alpha$ is real-analytic on $[0,1]$, then $\hat\Lambda_{\alpha/k}$ is real-analytic on $[0,1]$.
The next proposition provides conditions on $\Lambda$ that guarantee that the compactification defined through $\Theta$ in \eqref{eq:2Thetadefinition} is real-analytic.

\begin{prop}
\label{prop:analyticity_logistic_exact}
Let $\alpha>0$ and consider $\Theta$ as in \eqref{eq:2Thetadefinition}.
Then, $\hat\Lambda$ is real-analytic on $[0,1]$ if and only if the following two conditions hold:
\begin{itemize}
\item[{\rm(i)}] $\Lambda$ is real-analytic on $\R$;
\item[{\rm(ii)}] there exist $T_-<0<T_+$ and real-analytic functions $f_-$ and $f_+$ defined in neighborhoods of $0$ such that
\[
\Lambda(t)=f_-(e^{\alpha t}) \quad \text{for } t<T_-\,,
\qquad
\Lambda(t)=f_+(e^{-\alpha t}) \quad \text{for } t>T_+\,.
\]
\end{itemize}
In that case,
\[
\lb_-:=f_-(0)=\lim_{t\to-\infty}\Lambda(t)\,,
\qquad
\lb_+:=f_+(0)=\lim_{t\to+\infty}\Lambda(t)\,.
\]
\end{prop}

\begin{proof}
Assume first that $\hat\Lambda$ is real-analytic on $[0,1]$. Since $\Theta$ is real-analytic from $\R$ into $(0,1)$, we have $\Lambda=\hat\Lambda\circ\Theta$, hence $\Lambda$ is real-analytic on $\R$.

Moreover, since $\hat\Lambda$ is analytic at $0$, one gets
\[
\Lambda(t)=\hat\Lambda(\Theta(t))=\hat\Lambda\left(\frac{e^{\alpha\,t}}{1+e^{\alpha\,t}}\right)\,,
\]
so $f_-(z):=\hat\Lambda(\frac{z}{1+z}\bigr)$, which is analytic near $0$.
Similarly, $f_+(z):=\hat\Lambda\!\bigl(\frac1{1+z}\bigr)$ is analytic near $0$. This proves (ii). The identities $\lb_-=f_-(0)$ and $\lb_+=f_+(0)$ are immediate.

Conversely, assume (i) and (ii). Since $\Lambda$ and $\Theta^{-1}$ are real-analytic on $(0,1)$, the function $\hat\Lambda(\tau)=\Lambda(\Theta^{-1}(\tau))$ is real-analytic on $(0,1)$.

Near $\tau=0$, using $\exp(\alpha\,\Theta^{-1}(\tau))=\tau/(1-\tau)$, condition (ii) gives
\[
\hat\Lambda(\tau)=f_-\!\left(\frac{\tau}{1-\tau}\right)
\]
for $\tau>0$ small. The right-hand side is analytic at $\tau=0$, and its value at $0$ is $f_-(0)=\lb_-$. Hence, $\hat\Lambda$ is analytic at $0$.
The proof is analogous close to $\tau=1$.
\end{proof}

\medskip

{\bf The compactified system.}
Once a proper choice of $\Theta$ is done, that is, assuming that the hypotheses
of Proposition \ref{prop:2smoothness} hold, we define the $C^r$ map
\begin{equation}\label{eq:torecall1definitionG}
G\colon[0,1]\times\R^d\to\R^d,\quad G(\tau,\bx)=g(\hat\Lambda(\tau),\bx)\,,
\end{equation}
and translate \eqref{eq:2trivial_skewproduct} to a dynamical system on
$[0,1]\times\R^d$, generated by the iterations of the map
\begin{equation}\label{eq:compactified}
\begin{split}
F\colon [0,1]\times\R^d&\to [0,1]\times\R^d\\
\left(\tau,\bx\right)&\mapsto \left(A(\tau)\,,\;G(\tau,\bx)\right)\,,
\end{split}
\end{equation}
where the planes $\{\tau=0\}$ and $\{\tau=1\}$ are invariant.

Recall that these hyperplanes encode the system's past and future limit
dynamics. Under the previously considered assumptions on the compactification
framework, they are normally hyperbolic attracting/repelling hyperplanes. 
Because they play a role in tipping phenomena (see Section~\ref{subsec:phaseinducedtipping}),
we are interested in invariant compact subsets of the limit systems and in the
foliations of their invariant stable and unstable manifolds.
To construct explicit examples with concrete tipping properties, we focus  on
invariant tori of the limit systems, since there are well-established methods to
compute them and their invariant manifolds.

Concretely, for $m<d$ and a frequency vector $\w\in\mathbb{R}^m$,
let $T_\w$ be the rigid rotation  on $\T^m$ of angle $\w$, that is
$T_\w(\theta)=\theta+\w$ with $\w \in \T^m$.
For concreteness, we consider an invariant torus of the past equation \eqref{eq:2autonomous_limits} that is,
we consider an $m$-dimensional embedded torus $\mK\subseteq\R^d$
bearing internal dynamics conjugated to $T_\omega$
and invariant with respect to  $\bx\mapsto G(0,\bx) = g(\lb_-,\bx)$.
The last invariance condition means that
\[
\mK_-:=\{0\}\times\mK
\]
is an invariant torus under $G$ (that is, of the extended system). 
Equivalently, 
the torus $\mK$ is  the image of an embedding
\[
K:\mathbb{T}^m \to \mathbb{R}^d\,,
\]
that is,  $\mK=K(\T^m)$,  that is a solution of the invariance equation
\[
G(0,\cdot)\circ K = K\circ T_\omega\,.
\]

The geometric mechanism behind the next result, which is the foliation of the
unstable manifold, already appears in Fenichel's rate-condition theory
\cite{fenichel1974}.
A particularly clear flow-based exposition in the continuous case  can be found in
\cite[Theorem 5.6.1]{wiggins}, \cite[Theorem 4.1]{pesin} and \cite[Theorem 4.3]{nippstoffer}.
See~\cite{hirsch1977} for a classical approach to the  discrete-time normally
hyperbolic invariant manifold  theory.

\begin{teor}\label{teor:2main_result}
Let $\ell\geq1$, and let
\[
F:[0,1]\times\mathbb R^d\to[0,1]\times\mathbb R^d,
\qquad
F(\tau,x)=(A(\tau),G(\tau,x)),
\]
be a $C^{\ell+1}$ diffeomorphism. Set $g_0(x):=G(0,x)$.
Assume that:

\begin{enumerate}
\item[\rm(H1)] $A(0)=0$ and $A'(0)>1$.

\item[\rm(H2)] $\mK\subset\mathbb R^d$ is a compact $g_0$-invariant $C^{\ell+1}$
embedded manifold.
\item[\rm(H3)] There exists a neighbourhood
$\mV\subset\mathbb R^d$ of $\mK$ such that the $\upalpha$-limit set of every point of $\mV\setminus\mK$ is not contained in $\mK$.
\item[\rm(H4)] There exist $\tau_0>0$ and a compact neighbourhood
$\mU\subset\R^d$ of $\mK$ such that 
$U_0=[0,\tau_0]\times\mU$
verifies
$F^{-1}(U_0)\subseteq U_0$ and
\[
\sup_{y\in\mU}
\|Dg_0(y)\|
\max\left\{
1,\,
\left(\sup_{y\in\mU}\|DF^{-1}(0,y)\|\right)^{\,\ell}
\right\}
<A'(0)\,.
\]
\end{enumerate}
Then, there exist $\delta\in(0,\tau_0]$ and a $C^\ell$ unstable fiber map $W^u:[0,\delta)\times\mK\to\mathbb R^d$ such that
\[
\Psi:[0,\delta)\times\mK\to[0,\delta)\times\mathbb R^d,
\qquad
\Psi(\tau,x):=\bigl(\tau,\,W^u(\tau,x)\bigr)\,,
\]
is a $C^\ell$ embedding.
If we define 
$\mK_-:=\{0\}\times\mK$ and
\[
\mW^u_{\mathrm{loc}}(\mK_-)
:=
\Bigl\{
\bigl(\tau,W^u(\tau,x)\bigr)
\;\Big|\;
\tau\in(0,\delta),\ x\in\mK
\Bigr\}\,,
\]
the following properties hold:

\begin{enumerate}
\item[\rm(i)] For every $x\in\mK$, $W^u(0,x)=x$.
\item[\rm(ii)] For every
$\tau\in(0,\delta)$ and every $x\in\mK$
\begin{equation}\label{eq:Winvariancequation}
F^{-1}\bigl(\tau,W^u(\tau,x)\bigr)
=
\bigl(A^{-1}(\tau),W^u(A^{-1}(\tau),g_0^{-1}(x))\bigr)\,.
\end{equation}
\item[\rm(iii)] $\displaystyle{
\mW^u_{\mathrm{loc}}(\mK_-)
=
\bigl\{
z\in(0,\delta)\times\mathbb R^d
\;\big|\;
\lim_{n\to\infty}\dist\bigl(F^{-n}(z),\{0\}\times\mK\bigr)=0
\bigr\}}$.

\item[\rm(iv)] For every $\tau\in(0,\delta)$, the map
$\mK\to[0,1]\times\R^d$, $x\mapsto\bigl(\tau,W^u(\tau,x)\bigr),$
is a $C^\ell$ diffeomorphism onto its image.
\end{enumerate}

Moreover, $W^u$ is unique among $C^\ell$ maps satisfying the above
dynamical characterization.

In addition, if $\mK$ is a torus and the dynamics in $(\mK,g_0)$ is conjugate to
a rigid rotation, then for any $(\tau,x)\in(0,\delta)\times\mK$, $W^u(\tau,x)$
is the unique element $y\in\R^d$ satisfying
\begin{equation}\label{eq:asymptoticpair}
\lim_{n\to\infty}\dist\Big(F^{-n}\big(\tau,y\big)\,,\big(0,g_0^{-n}(x)\big)\Big)=0\,.
\end{equation}
\end{teor}
\begin{proof} 
See Appendix~\ref{appendix:proofoftheorem}.
\end{proof}

\medskip

The previous result follows from well-established results on fiber bundles already present in
the literature,  adapted to the particular compactified skew-product structure of our map.
While we relegate the more technical details of the proof to Appendix~\ref{appendix:proofoftheorem}, the following list of remarks might help clarify the role of the assumptions of the theorem and indicate how the general theory is adapted to the particular setting considered here.

\begin{remark}
The general normal hyperbolicity theory establishes the existence of unstable fibers associated with a normally hyperbolic invariant manifold and, under additional bunching assumptions, gives regularity of the dependence of these fibers on the base point.
The reason is that, although the leaves of the foliation are as smooth as the map itself, the dependence of these leaves on the point $x\in\mK$ is more delicate and is constrained by the ratio between the normal and tangential exponents to the normally hyperbolic invariant manifold.
In the general setting, one must also impose conditions to guarantee the required regularity of the invariant manifold itself.
In our case, however, this latter issue does not arise: the invariant manifold is explicitly given, independently of any further bunching condition.
Thus, the usual assumptions involving bounds such as powers of $\|Dg_0(y)\|$ that are needed in the general theory to control the regularity of the invariant manifold are not required here.
Instead, the bunching condition in {\rm(H4)} is used only to make the graph transform for the unstable direction contractive in the $C^\ell$ topology.
This yields a $C^\ell$ dependence of the unstable fibers on the base point $x\in\mK$.

Finally, let us point out that the existence of the backward invariant
neighborhood in {\rm(H4)} is an instrumental assumption.
It is imposed only to have a space on which the graph transform is defined. If
this condition is not satisfied by the
original map, one may perform a smooth cutoff outside a sufficiently small
neighborhood of $\{0\}\times\mK$, without changing the map near
$\{0\}\times\mK$. The modified map then satisfies {\rm(H4)}, and the local
unstable fibers obtained for it coincide, after possibly shrinking the
neighborhood, with those of the original map.

\end{remark}
\begin{remark}
    In the case where $\mK$ is a uniformly stable attractor for the map $g_0$, the results already present in the literature \cite{alk2018,kiers2020,wieczorek2021} establish the existence of a local pullback attractor for the full system, which in that case coincides with the manifold $\mW_\mathrm{loc}^u(\mK_-)$.
    The usefulness of the present approach lies in the fact that hypotheses (H3) and (H4) are weaker, allowing for the existence of this manifold, which is sectionwise homeomorphic to $\mK$, even in situations where no local pullback attractor exists.
    Of particular relevance for our purposes is the case where $\mK$ is a torus belonging to a foliation of tori in an area-preserving conservative system (see examples in Section~\ref{sec:32examples_reversibilities}).
\end{remark}
\begin{remark}\label{rmk:2mK_+}
    An entirely analogous construction can be carried out at the opposite end $\tau=1$, replacing $g_0$ by the $d$-dimensional reduced map $g_1(x):=G(1,x)$ and assuming that $A(1)=1$ and that the $\tau$-direction is normally contracting there, i.e. $A'(1)<1$.
    In that case one obtains a local stable manifold attached to a $g_1$-invariant torus $\mK\subset\R^d$, contained in a strip $(1-\delta,1)\times\R^d$ and characterized by forward convergence to $\mK_+:=\{1\}\times\mK$. The translation of {\rm(H3)} is then obtained by replacing $\upalpha$-limits with $\upomega$-limits, namely one requires that the $\upomega$-limit set of every point of $\mN\setminus\mK$ have empty intersection with $\mK$. Likewise, {\rm(H4)} is replaced by the corresponding domination and bunching conditions.
\end{remark}
\begin{remark}\label{remark:asymptoticpairs}
The property  \eqref{eq:asymptoticpair}  states that
$(\tau,y)$ and $(0,x)$ are asymptotic pairs as time decreases under the action of $F$; that is, their orbits approach one another at very negative times and have the same dynamics.
Indeed, in order to obtain the characterization \eqref{eq:asymptoticpair} that is valid for a given $x\in\mK$ and any $\tau\in[0,\delta)$, one may replace the condition that $g_0$ is a rigid rotation on a torus by the condition that there are no points in $\mK$, other than $x$, that are asymptotic pairs with $x$ as time decreases under the action of $g_0$.
\end{remark}
\begin{remark}\label{rmk:2globalization}
    Assumption (H1) expresses the uniform hyperbolicity of the hyperplane $\{\tau=0\}$.
    This uniformity is essential for obtaining the constant $\delta>0$, independent of $x\in\mK$, such that $W^u(\tau,x)$ is well defined for all $(\tau,x)\in[0,\delta)\times\mK$.
    Next, the invariance equation \eqref{eq:Winvariancequation} implies that
    \begin{equation}\label{eq:2howtoextend}
    W^u(\tau,x)=G\big(A^{-1}(\tau),W^u(A^{-1}(\tau,g_0^{-1}(x)))\big)
    \end{equation}
    holds for all $(\tau,x)\in[0,\delta)\times\mK$.
    We may therefore use \eqref{eq:2howtoextend} as a definition wherever its right-hand side is previously defined.
    Proceeding in this way, we define 
    \[
    \tau_u:=\sup\big\{\tau\in[0,1]\mid\,W^u(\bar\tau,x)\text{ is defined for all }(\bar\tau,x)\in[0,\tau]\times\mK\big\}\,.
    \]
    as the supremum of the values of $\tau$ for which $W^u(\tau,x)$ is defined for every $x\in\mK$.
    Then, for any $\delta$  such that $0<\delta<\tau_u$, we can define the globalized unstable manifold 
\[
\mW^u(\mK_-):=\big\{F^n(\tau_0,W^u(\tau_0,x_0))\mid\, n\in\N, (\tau_0,x_0)\in(0,\delta)\times\mK\big\}\,.
\]
Note that, for any $(\tau,x)\in\mW^u(\mK_-)$, there exists $x_0\in\mK$ such that $x=W^u(\tau,x_0)$. 
\end{remark}

\section{Phase-Induced Tipping and Reversibility Constraints}\label{subsec:phaseinducedtipping}

This section discusses two approaches to phase-induced tipping in asymptotically autonomous systems that have been proposed in the literature and examines how reversibility can prevent this phenomenon.

\subsection{Two approaches to phase-induced tipping}
In this subsection, we present two different approaches to tipping induced by
phase-change that are found in the literature, although often formulated there
for ODEs, and explain them in a unified manner through the unstable foliation
Theorem~\ref{teor:2main_result} introduced in the previous section.
However, it is worth noting that some results in the literature on this type of tipping are developed in a more general framework than the one considered in this article, where the dynamical system cannot be necessarily expressed as asymptotically autonomous \cite{dno4,hoyerleitzel2021,longonunezobaya,lnor}

In any case, the points of the unstable manifold $\mW^u(\mK_-)$ (recall
Remark~\ref{rmk:2globalization}) are understood to represent the evolution of
certain states that originate from $\mK_-$ in the past.
In this way, the two approaches presented below concern the identification of
different possible $\upomega$-limit sets for distinct points of the form
\[
\big(\tau,\,W^u(\tau,x)\big)
\]
for $(\tau,x)\in\mW^u(\mK_-)$.
That is,  one has to study the possible  limits, as $n\to\infty$, of sequences (orbits) of the type
\[
\big(F^n(\tau,W^u(\tau,x))\big)_n \, .
\]

First, let us make precise which recurrence equation solves the second component of the previous sequence, which takes values in $\R^d$, as this information will prove useful in explaining the two different approaches to phase-induced tipping below.
Equation~\eqref{eq:Winvariancequation} implies that
$F^n(\tau,W^u(\tau,x))=(A^n(\tau),W^u(A^n(\tau),g_0^n(x))$,
and then 
\[
\begin{split}
F^{n+1}(\tau,W^u(\tau,x))&=F\big(F^n\big(\tau,W^u(\tau,x)\big)\big)\\
&=\Big(A^{n+1}(\tau),\,G\big(A^n(\tau),W^u\big(A^n(\tau),g_0^n(x)\big)\big)\Big)\,,
\end{split}
\]
Since $\tau\in(0,1)$,
it follows from equations \eqref{eq:torecall2definitionLambdaHat} and
\eqref{eq:torecall3definitionA} that
\[
\hat\Lambda(A^n(\tau))=\Lambda\big(\Theta^{-1}(A^n(\tau))\big)=\Lambda(n+\Theta^{-1}(\tau))\,,
\]
and thus \eqref{eq:torecall1definitionG} yields
\[
F^{n+1}(\tau,W^u(\tau,x))=\big(A^{n+1}(\tau),\,g\big(\Lambda(n+\Theta^{-1}(\tau)),W^u(A^n(\tau),g_0^n(x))\big)\big)\,.
\]
Therefore, using \eqref{eq:Winvariancequation} on the left-hand side we get
\[
W^u(A^{n+1}(\tau),g_0^{n+1}(x))=g\big(\Lambda(n+\Theta^{-1}(\tau)),W^u(A^n(\tau),g_0^n(x))\big)\,,
\]
that is, the sequence $\big(W^u(A^{n}(\tau),g_0^{n}(x))\big)_n$ is a solution of the recurrence
\begin{equation}\label{eq:3recurrenceagain}
x_{n+1}=g(\Lambda(n+c),x_n)\,,\quad n\in\Z\,,
\end{equation}
with $c=\Theta^{-1}(\tau)$.
Consequently, for any $\tau_1,\tau_2\in(0,1)$ for which there exists $m\in\Z$ such that $\tau_1=A^m(\tau_2)$, we conclude that the second component of 
\begin{equation}
    \label{eq:backtotherecurrence}
(F^n(\tau_1,W^u(\tau_1,x)))_n
\quad\text{and}\quad
(F^n(\tau_2,W^u(\tau_2,x)))_n
\end{equation}
satisfy the same recurrence equation, with the index shifted by $m$.
Otherwise, the second components of \eqref{eq:backtotherecurrence} do not satisfy the same recurrence equation \eqref{eq:3recurrenceagain} for any fixed constant value of $c \in \mathbb{R}$,
since subsequent iterates cannot be related by an integer shift of the index. 
To consider all the possible recurrence equations arising in \eqref{eq:backtotherecurrence}, it suffices to study initial values of $\tau$ varying in a fundamental domain $[\bar\tau,A(\bar\tau))$ for certain $\bar\tau\in(0,\delta)$ such that $A(\bar\tau)\in(0,\delta)$.
\medskip

In the following, we assume that the characterization~\eqref{eq:asymptoticpair} holds, that is,  that the points of the form $(\tau,W^u(\tau,x))$ on the manifold $\mW^u(\mK_-)$ are the only points within $[0,1]\times\R^d$ that are asymptotic pairs with $(0,x)$ as time decreases under the action of $F$ (recall Remark~\ref{remark:asymptoticpairs}).
Then, taking into account the previous information,
we summarize the approaches to tipping due to phase change in the following two:
\medskip

\textbf{First approach.}
For some fixed $\bar\tau\in(0,\tau_u)$ such that $A(\bar\tau)\in(0,\tau_u)$, we study the $\upomega$-limit sets of $\big(\tau,\,W^u(\tau,x)\big)$ for all $x\in\mK$, either by taking $\tau=\bar\tau$ fixed or by allowing $\tau$ to vary over a fundamental domain $[\bar\tau,A(\bar\tau))$.
Then, all the orbits under consideration are sequences of the form
\[
\big(F^n(\tau,W^u(\tau,x))\big)_n
\]
for $x$ varying in $\mK$ and either $\tau=\bar\tau$ or $\tau\in[\bar\tau,A(\bar\tau))$ depending on the recurrence equations under consideration.
Tipping occurs when there are multiple distinct $\upomega$-limit sets (that is, more
than one possible long-term state) for the system.

This approach can be interpreted as follows: orbits asymptotically approaching
certain regions of the past invariant set $\mK_-$ converge to $\upomega$-limit
sets different from those reached by orbits originating in other regions.
This corresponds to the phase-tipping viewpoint of \cite{alk2018} and
\cite{alk2021}, where tipping depends on the initial state on the past
attractor.
\medskip

\textbf{Second approach.}
For some fixed $\bar x\in\mK$, we study the $\upomega$-limit sets of
$\big(\tau,\,W^u(\tau,\bar x)\big)$ either by taking
$\tau\in\{A^n(\bar\tau)\mid\,n\in\Z\}\cap(0,\tau_u)$ or $\tau\in(0,\tau_u)$.
Again, the first option amounts to restricting ourselves to solutions of the
same recurrence equation \eqref{eq:3recurrenceagain}, up to an integer shift of
the index, whereas the second option amounts to considering solutions of
\eqref{eq:3recurrenceagain} for all possible real values of $c$, as explained
above.

Within this approach, studying tipping amounts
to fixing an element $x$ of the past invariant compact set $\mK_-$
and  analyzing all possible $\upomega$-limits of  points lying on the fiber of its
unstable foliation.  This approach arises naturally when the phase
is allowed to vary as an additional parameter  in the model under consideration,
entering additively over time into certain terms of the equation.
This approach is close in spirit to that of \cite{dno3,dno4}.
\medskip

Naturally, the suitability of one approach or the other will depend on the nature of the problem under study.
The next proposition establishes a strong link between these two approaches,
showing that the sets under study in both approaches are dense in one another.

\begin{prop}Assume that $(\mK,g_0)$ is minimal.
Let $\bar x\in\mK$ be fixed.
Then,
for any $\bar\tau\in(0,\tau_u)$,
\begin{equation}
    \mF_{\bar\tau,\bar x}:=\big\{ F^n(\tau,W^u(\tau,\bar x))\;\big|\;n\in\N\,,\;A^n(\tau)=\bar\tau\,,\;\tau\in(0,\min\{\delta,\bar\tau\})\big\}
\end{equation}
is dense in
\[
\mW_{\bar\tau}^u:=\big\{(\bar\tau,W^u(\bar\tau,x))\,\mid\,x\in\mK\big\}\,.
\]
Consequently,
\[
\big\{ F^n(\tau,W^u(\tau,\bar x))\;\big|\;n\in\N\,,\;A^n(\tau)\in[\bar\tau,A(\bar\tau))\,,\;\tau\in(0,\min\{\delta,\bar\tau\})\big\}
\]
is dense in
\[
\big\{(\tau,W^u(\tau,x))\,\mid\,\tau\in[\bar\tau,A(\bar\tau))\,,\;x\in\mK\big\}\,.
\]
\end{prop}
\begin{proof}
The invariance property~\eqref{eq:Winvariancequation} can be rewritten as
\[
F^n\bigl(\tau,W^u(\tau,x)\bigr)
=
\bigl(A^n(\tau),\,W^u(A^n(\tau),g_0^n(x))\bigr) \,,
\]
and therefore
\[
\mF_{\bar\tau,\bar x}=\big\{\big(\bar\tau,W^u(\bar\tau,g_0^n(\bar x))\big)\,\mid\; n\in\N\big\}\,.
\]
Since $(\mK,g_0)$ is minimal, every forward orbit is dense in $\mK$.
This fact and the continuity of $W(\tau,x)$ proves the first claim.
The second one follows from the first one, since $\bigcup_{\tau\in[\bar\tau,A(\bar\tau))}\mF_{\tau,\bar x}$ is dense in $\bigcup_{\tau\in[\bar\tau,A(\bar\tau))}\mW_\tau^u$.
\end{proof}

The preceding proposition shows that, within the asymptotically autonomous setting considered here,
the second approach provides a subset of the information offered by the first.
It is worth emphasizing, however, that if the asymptotic dynamics were fully
nonautonomous, it would not even be clear how to define the first approach,
whereas the second approach would still admit a natural definition and remains
applicable. Consequently, these two approaches are complementary tools, with the
relative usefulness determined by the structure of the asymptotic dynamics.

In Section~\ref{sec:42surgicalintersections}, we present an example in which the intersection between the unstable manifold of the object selected in the past and the stable manifold of the object selected in the future consists of a single orbit.
Equivalently, when restricted to a single fundamental domain, this intersection reduces to a single point.
Depending on the fibre chosen in the second approach, this point may not belong to the dense set $\mF_{\bar\tau,\bar x}$.
Consequently, the second approach would not detect any change in the behaviour of the $\upomega$-limit sets under consideration, whereas first approach would detect such a change.
This example highlights the importance of selecting an appropriate approach to tipping induced by change of phase depending on the specific problem under investigation.

%SECTION ON REVERSIBILITY
\subsection{Reversibility and time-dependent modulation}
In this section,
under the assumption that past and future systems are the same,
we present reversibility-based arguments
that preclude the existence of distinct $\upomega$-limit sets within $\{\tau=1\}$ for points on the unstable manifold $\mW^u(\mK_-)$, or at least ensure that some points on that manifold converge towards a known $\upomega$-limit set.

To carry out this study, Section~\ref{rev_f}
focuses  on a specific type of nonautonomous variation of the map \eqref{eq:compactified}, namely a time-dependent modulation of a $C^r$ reversible map $f:\R^d\to\R^d$.
More precisely, given a function $\hat\Lambda\colon[0,1]\to\R$ satisfying $\hat\Lambda(0)=\hat\Lambda(1)=1$, we study the dynamics induced by a map of the form
\begin{equation}\label{eq:2modulationmap}
\begin{split}
F\colon[0,1]\times\R^d&\to[0,1]\times\R^d\\
(\tau,x)&\mapsto(A(\tau),\,\hat\Lambda(\tau)f(x))\,.
\end{split}
\end{equation}
This class of maps provides a natural framework for modeling nonautonomous discrete dynamics in which the intrinsic evolution of the system is described by the map $f$, while the factor $\hat\Lambda(\tau)$ represents a time-dependent expansion or contraction.
Such a modulation may arise from external gain or loss, parametric forcing, environmental variability, seasonal effects, or control mechanisms acting on the system over time.
In this kind of systems, we address the question of whether the reversibility property of the map $f$ can be extended to a spatio-temporal reversibility property of the map $F$, and what implications this reversibility has for the $\upomega$-limit sets of points in $\mW^u(\mK_-)$.
On the contrary, in Section~\ref{rev_F} we do not assume such a specific type of compactified system. Instead, we discuss how imposing restrictive reversibility properties
on $G(\tau,x)$ implies the existence of heteroclinic connections between the past and the future systems, and whether we can obtain a smooth cylinder formed by connections in the discrete setting.
Finally, in Section~\ref{sec:32examples_reversibilities} we illustrate the properties derived from the reversibilities in both settings. 

\subsubsection{Reversible limit map $f$.} \label{rev_f}
Consider a compactified system of the form~\eqref{eq:2modulationmap}.
We recall that $f$ is said to be \emph{reversible} if $f$ is invertible and there exists a map $S\colon\R^d\to\R^d$, called a reversing symmetry or reversor of $f$, satisfying $S^2=\mathrm{Id}$, that is, $S$ is an involution--and $f=S\circ f^{-1}\circ S$.
Note that $f^{-1}\circ S$ is also a reversor of $f$, so every reversible map can be writen as the composition of two reversions.
Moreover, every $C^1$ reversible planar map can be conjugated through a $C^1$ difeomorphism to a map which has a linear reversor~\cite{cima2015}.

\begin{prop}\label{prop:2reversibility} Assume that
\begin{itemize}
\item[{\rm(i)}] $f$ is reversible and admits a linear reversor $L$ such that $f=L \circ S$,
\item[{\rm(ii)}] $A(\tau)+A^{-1}(1-\tau)=1$ for all $\tau\in[0,1]$,
\item[{\rm(iii)}] $\hat\Lambda(\tau)\,\hat\Lambda(1-A(\tau))=1$ for all $\tau\in[0,1]$.
\end{itemize}
Then, $F$ admits a reversor $R(\tau,x)=(1-\tau,S(x))$.
\end{prop}
\begin{proof}
First, note that
\[
F^{-1}(\tau,x)=\left(A^{-1}(\tau),\;f^{-1}\left(\frac{x}{\hat\Lambda(A^{-1}(\tau))}\right)\right)\,.
\]
Then, direct computation shows that
\[
    R\circ F^{-1}\circ R\,(\tau,x)
    =\left(1-A^{-1}(1-\tau),\;S\circ f^{-1}\left(\frac{S(x)}{\hat\Lambda(A^{-1}(1-\tau))}\right)\right)\,.
\]
So, since (ii) shows that $A(\tau)=1-A^{-1}(1-\tau)$, it suffices to check that
\[
\hat\Lambda(\tau)f(x)=S\circ f^{-1}\left(\frac{S(x)}{\hat\Lambda(A^{-1}(1-\tau))}\right)\,.
\]
Since $S\circ f^{-1}=L$ is linear,
\[
S\circ f^{-1}\left(\frac{S(x)}{\hat\Lambda(A^{-1}(1-\tau))}\right)=\frac1{\hat\Lambda(A^{-1}(1-\tau))}\,S\circ f^{-1}\circ S(x)=\frac{f(x)}{\hat\Lambda(A^{-1}(1-\tau))}\,,
\]
and assumptions (ii) and (iii) ensure that $f(x)/\hat\Lambda(A^{-1}(1-\tau))=\hat\Lambda(\tau)f(x)$.
\end{proof}
\begin{remark}\label{remark:3conditions(i)and(ii)involutive}
    1. The condition (ii) holds whenever the diffeomorphism $\Theta$ satisfies the symmetry $1=\Theta(t)+\Theta(-t)$ for all $t\in\R$, that is, the map $t\mapsto 1/2-\Theta(t)$ is odd.
    This is, in particular, the case of \eqref{eq:2Thetadefinition}.
   \smallskip
   
    2. We briefly discuss here a simple way to find functions $\hat\Lambda$ satisfying (iii), once (ii) is assumed.
    First, note that (ii) ensures that $B(\tau)=1-A(\tau)$ is an involution since
    \[
    B^2(\tau)=1-A(1-A(\tau))=1-A(A^{-1}(1-\tau))=\tau\,,
    \]
    and satisfies $B(0)=1$ and $B(1)=0$ because of \eqref{eq:torecall3definitionA}.
    Now, choose an antisymmetric function $Q(\tau_1,\tau_2)$, that is, a function such that $Q(\tau_1,\tau_2)=-Q(\tau_2,\tau_1)$, and define
    \begin{equation}\label{eq:3hatLambda_definition}
    \hat\Lambda(\tau)=\exp Q(\tau,B(\tau))\,.
    \end{equation}
    Then, the involutive character of $B$ and the antisymmetric character of $Q$ yield
    \[
    \hat\Lambda(\tau)\hat\Lambda(1-A(\tau))=\hat\Lambda(\tau)\hat\Lambda(B(\tau))=\exp\big(Q(\tau,B(\tau))+Q(B(\tau),B^2(\tau))\big)=1\,.
    \]
     We also have to assume that $Q(0,1)=0$ to guarantee that $\hat\Lambda(0)=\hat\Lambda(1)=1$.
    \end{remark}

Note that hypothesis (i) of Proposition \ref{prop:2reversibility} is present in some frequently used models, for example:
\begin{itemize}
    \item[{\rm(1)}]  the generalized standard maps
    \[
    f(x,y)=(x+y+F(x),\,y+F(x))\,,
    \]
    where $F\colon\R^d\to\R^d$,  can be written as  $f=L\circ S$ with
    $L(x,y)=(y-x,y)$  and $S(x,y)=(-x,y+F(x))$,  where $S$ is a reversor if $F$ is odd.
    \item[{\rm(2)}]  area-preserving generalized Hénon maps,
    \[
    f(x,y)=(y,\,-x+F(y))\,.
    \]
    In this case, one can take $ L(x,y)=(y,x)\,,$ and write $f=L\circ S$, where $S(x,y)=(-x+F(y),y)$.
\end{itemize}
The same conclusion holds for any map which is linearly topologically conjugate to one of the maps in {\rm(1)} or {\rm(2)}.

\medskip

The following result explains the relevance of the reversibility constructed in Proposition \ref{prop:2reversibility} through the existence of heteroclinic orbits, which will generically be part of one-dimensional families of heteroclinic orbits connecting each invariant torus of the past system with the same invariant torus of the future system.
 We note that condition $\hat\Lambda(0)=\hat\Lambda(1)=1$ guarantees $F(0,x)=F(1,x)$ for all $x\in\R^d$, and hence the invariant objects of the past and future systems can be identified without ambiguity.
\begin{coro}
\label{cor:3heteroclinicorbits}
Assume that the hypotheses of Proposition \ref{prop:2reversibility} hold.
Then,
\begin{enumerate}
    \item[{\rm (i)}] Given any subset $\mC\subseteq\{1/2\}\times\mathrm{fix}(S)$, the $F$-invariant set
    \[
    \bigcup_{n\in\Z}F^n(\mC)
    \]
    is invariant under the reversor $R$ given in Proposition \ref{prop:2reversibility}.
    \item[{\rm (ii)}] For any $(\tau,x)\in\{1/2\}\times\mathrm{fix}(S)$, the $\upalpha$-limit set and $\upomega$-limit set of $(\tau,x)$ for $F$ are of the form $\{0\}\times\mB$ and $R(\{0\}\times\mB)=\{1\}\times S(\mB)$ for some $f$-invariant compact set $\mB\subset\R^d$, respectively.
    If, in addition, $\mB$ is minimal under $f$ and $\mB\cap\mathrm{fix}(S)\neq\emptyset$, $\mB=S(\mB)$, so the $\upomega$-limit set of $(\tau,x)$ is $\{1\}\times\mB$.
\end{enumerate}    
\end{coro}
\begin{proof}
(i)
First, note that $\mC=\{1/2\}\times\mD$ for some $\mD\subseteq\mathrm{fix}(S)$.
Therefore,
\[
R(\mC)=\{1/2\}\times S(\mD)=\{1/2\}\times\mD=\mC\,,
\]
that is, $\mC$ is invariant under $R$.
Then,
\[
R(F^n(\mC))=F^{-n}(R(\mC))=F^{-n}(\mC)\,,
\]
from where we conclude that
\[
R\left(\bigcup_{n\in\Z}F^n(\mC)\right)=\bigcup_{n\in\Z}R\big(F^n(\mC)\big)=\bigcup_{n\in\Z}F^{-n}(\mC)=\bigcup_{n\in\Z}F^n(\mC)\,,
\]
so the set $\cup_{n\in\Z}F^n(\mC)$ is invariant under $R$, as we wanted to see.
\medskip

(ii) The first assertion follows easily from $R\circ F^n(1/2,x)=F^{-n}\circ R(1/2,x)=F^{-n}(1/2,x)$.
Assume now that $\mB$ is minimal under $f$ and that $\mB\cap\mathrm{fix}(S)\neq\emptyset$ and note first that $f(S(\mB))=S(f^{-1}(\mB))=S(\mB)$, so $S(\mB)$ is invariant under $f$ and it is closed since $S$ is an involution.
Take $y\in \mB\cap\mathrm{fix}(S)$ and note that $y\in\mB\cap S(\mB)$.
Since $\mB$ is minimal, it is compact and the orbit of $y$ is dense in $\mB$.
Thus, $\mB\subseteq S(\mB)$, and since $S$ is an involution, it follows that $\mB=S(\mB)$, as we wanted to see.
\end{proof}
\begin{remark}
    Let $\mK\subset\R^d$ be an $m$-dimensional embedded invariant torus for $f$ with internal dynamics conjugated to an irrational rotation and such that $\mK\cap\mathrm{fix}(S)\neq0$.
    Consider the manifold $\mW^u(\mK_-)$, assume $1/2<\tau_u$ (recall Remark \ref{rmk:2globalization}), define
    \[
    \Sigma_{1/2}:=\big\{W^u(1/2,x)\mid\,x\in\mK\big\}\subset\R^d
    \]
    and assume that $\Sigma_{1/2}\cap\mathrm{fix}(S)\neq\emptyset$.
    
    Note that Theorem \ref{teor:2main_result}(iii) ensures that, for any $x\in\Sigma_{1/2}$, the $\upalpha$-limit set of $(1/2,x)$ is contained in $\mK_-$.
    Hence, by the minimality of $\mK$ under $f$, this $\upalpha$-limit set must in fact coincide with $\mK_-$.
    Therefore, for each $x\in\Sigma_{1/2}\cap \mathrm{fix}(S)$, Corollary \ref{cor:3heteroclinicorbits}(ii) ensures that the $\upomega$-limit set of $(1/2,x)$ is $\mK_+:=\{1\}\times\mK$.
    
    That is, each point in $\Sigma_{1/2}\cap \mathrm{fix}(S)$ provides an heteroclinic connection between the tori $\mK_-=\{0\}\times\mK$ and $\mK_+=\{1\}\times\mK$.
    Assume that $\mathrm{fix}(S)$ is an embedded submanifold of $\mathbb R^d$ of dimension $k$.
    One has $\dim\Sigma_{1/2}=\dim\mK=m$.
    Then, if $\mathrm{fix}(S)$ intersects $\Sigma_{1/2}$ transversely in $\mathbb R^d$, their intersection is a submanifold of dimension $\dim\bigl(\Sigma_{1/2}\cap\mathrm{fix}(S)\bigr)=m+k-d$.
    Each point of this manifold belongs, by the previous argument, to $\mW^u(\mK_-)\cap \mW^s(\mK_+)$.
    If, in addition, the intersection between these two invariant manifolds is transverse at every point of $\Sigma_{1/2}\cap\mathrm{fix}(S)$, then this gives rise to a continuum of heteroclinic trajectories between $\mK_-$ and $\mK_+$ which, after including the $\tau$-direction, has dimension $m+k-d+1$ on $[0,1]\times\R^d$.
\end{remark}

\subsubsection{Reversibilities on the compactified system leading to a continuum of heteroclinic orbits} \label{rev_F}
We finally state a much stronger condition on the general map \eqref{eq:compactified}, rather than on the particular setting considered in the previous subsection, guaranteeing that every orbit of $\mW^u(\mK_-)\cap\{\tau=1/2\}$ is an heteroclinic connection between $\{0\}\times\mK$ and $\{1\}\times\mK$.
In \cite{scarcella2024-2}, a related continuous time condition is given for asymptotically autonomous hamiltonian flows, hence providing such kind of connections, that in the continuous setting give rise to a KAM-type invariant cylinder.
In the present discrete formulation, however, this cylinder-like structure may no longer persist as a consequence of the splitting of the invariant manifolds,
as we will illustrate in the second example of Section \ref{sec:32examples_reversibilities}.

\begin{prop}\label{prop:3supersymmetry}
Assume that
\begin{enumerate}
    \item[{\rm(i)}] $x\mapsto G(\tau,x)$ is invertible for all $\tau\in[0,1]$,
    \item[{\rm(ii)}] $A(\tau)+A^{-1}(1-\tau)=1$ for all $\tau\in[0,1]$,
    \item[{\rm(iii)}]  there exists a closed $F$-invariant set $\mC\subseteq[0,1]\times\R^d$ such that $\mC\cap\big(\{1/2\}\times\R^d\big)\neq\emptyset$ and $x=G(1-A(\tau),G(\tau,x))$ for all $(\tau,x)\in\mC$.
\end{enumerate}
Then, $R(\tau,x)=(1-\tau,x)$ is a reversor for $F|_\mC$, and for any given point
$(1/2,x)\in\mC$ its $\upalpha$-limit and $\upomega$-limit sets (if they exist as bounded sets)
are subsets of $\mC$ of the form $\{0\}\times\mB_x$ and $\{1\}\times\mB_x $ for some set $\mB_x\subseteq\R^d$.
\end{prop}
\begin{proof} Let $G^{-1}\colon[0,1]\times\R^d\to\R^d$ be such that $G^{-1}(\tau,G(\tau,x))=\mathrm{Id}$.
Note that
\[
R\circ F^{-1}(\tau,x)=(1-A^{-1}(\tau),G^{-1}(A^{-1}(\tau),x))=(A(1-\tau),G(1-\tau,x))=F\circ R(\tau,x)\,,
\]
for all $(\tau,x)\in\mC$,
where $1-A^{-1}(\tau)=A^{-1}(\tau)$ follows from (ii) and $G^{-1}(A^{-1}(\tau),x)=G(1-\tau,x)$ follows from (iii).
    The assertion about the $\upalpha$-limit and $\upomega$-limit sets can be deduced from arguments analogous to those in the proof of Corollary \ref{cor:3heteroclinicorbits}.
\end{proof}

\subsubsection{Examples exhibiting reversibilities}\label{sec:32examples_reversibilities}
We illustrate the results of this section through the Hénon map \cite{lamb2026,vieirophd}.
To this end, we consider the maps
\[
T_a\left(\begin{matrix} x\\y\end{matrix}\right)=\left(\begin{matrix}x\\y-ax^2\end{matrix}\right)\,,\qquad
R_\beta\left(\begin{matrix}x\\ y\end{matrix}\right)=\left(\begin{matrix} \cos 2\pi\beta & -\sin 2\pi\beta\\ \sin 2\pi\beta & \cos 2\pi\beta \end{matrix}\right)\left(\begin{matrix}x\\ y\end{matrix}\right)\,,
\]
and the Hénon map $H_{\beta}=R_\beta\circ T_1$.
We refer to $a$ as the shear parameter and to $\beta$ as the rotation parameter.
A straightforward calculation yields
 $\|DH_\beta(x,y)\|_2=\sqrt{1+x^2}+|x|$, which will be used to verify 
 the gap condition in hypothesis (H4) of Theorem~\ref{teor:2main_result}.
\medskip

\textbf{First example.}
We consider the map
\begin{equation}\label{eq:3_modelfirstexample}
F(\tau,x,y)=\big(A(\tau),\,\hat\Lambda(\tau)\,H_\beta(x,y)\big)\,,
\end{equation}
for a fixed value $\beta=0.212$ of the rotation parameter, where 
the  map $A$ is given by \eqref{eq:2Adefinition} and
we define $\hat\Lambda$ through a simple antisymmetric function $Q(\tau_1,\tau_2)=20(\tau_1-\tau_2)\,\tau_1^2\,\tau_2^2$ as in \eqref{eq:3hatLambda_definition}.
Remark~\ref{remark:3conditions(i)and(ii)involutive}.2 guarantees that condition (iii) of Proposition \ref{prop:2reversibility} hold.
The map $H_\beta$ is reversible with $H_\beta=L_\beta\circ S$ and
\[
L_\beta\left(\begin{matrix} x \\y \end{matrix}\right)=\left(\begin{matrix} \cos(2\pi\beta) & \sin(2\pi\beta) \\ \sin(2\pi\beta) & -\cos(2\pi\beta) \end{matrix}\right)\left(\begin{matrix}x\\ y\end{matrix}\right)\,,\qquad
S\left(\begin{matrix} x\\ y\end{matrix}\right)=\left(\begin{matrix} x \\ x^2-y\end{matrix}\right)\,.
\]
 Let $(r,\varphi)$ denote polar coordinates around the origin, which 
 is an elliptic fixed point of $H_\beta$. 
 Since the last is (locally) a twist map, the invariant curves surrounding the origin are graphs over the angular variable $\varphi$ and, by fixing a rotation number, one identifies an invariant curve $\Gamma$ in the phase plane of $H_\beta$ (locally around the origin). Then, we can numerically approximate the curve $\Gamma$ by applying Newton’s method to the invariance equation, thereby obtaining a Fourier representation of the invariant curve in terms of the angular variable $\varphi$, see \cite{jorba2001}  (this is referred to as the large-matrix method in \cite{mamotreto}).
Note that by considering $H_\beta^k$, for $k \geq 2$, we can obtain a $k$-periodic invariant curve $\Gamma$ of $H_\beta$.
Since $\hat\Lambda'(0)=0$, after some calculations it follows that
\[
DF(F^{-1}(0,x,y))=\left(
\begin{array}{ccc}
e^\alpha & 0 & 0\\
0 & \cos(2\pi\beta)+2\sin(2\pi\beta)\xi & -\sin(2\pi\beta)\\
0 & \sin(2\pi\beta)-2\cos(2\pi\beta)\xi & \cos(2\pi\beta)\\
\end{array}
\right)\,,
\]
where $\xi=x\cos(2\pi\beta)+y\sin(2\pi\beta)$.
It can be checked that $\|DF^{-1}(0,x,y)\|_2=\max\{e^{-\alpha},\sqrt{1+\xi^2}+|\xi|\}$.
If one considers $\alpha=1.8$, as in the examples below,
condition (H4) of Theorem~\ref{teor:2main_result}
with $\ell=1$ can be guarantee to hold for all
$\mK\subseteq\{(x,y)\mid x^2+y^2\leq1\}$.
Although the above estimate provides a conservative bound, the corresponding quantity can be evaluated much more accurately by numerical means.

Note that, if $\Gamma$ is minimal, Corollary~\ref{cor:3heteroclinicorbits}(ii) guarantees that $S(\Gamma)=\Gamma$.
Therefore, defining $\Gamma_-:=\{0\}\times\Gamma$ and $\Gamma_+:=\{1\}\times\Gamma$, we obtain $\Gamma_+=R(\Gamma_-)$, where $R(\tau,x)=(1-\tau,S(x))$.
In addition, Corollary \ref{cor:3heteroclinicorbits}(ii) ensures the existence of heteroclinic connections between $\Gamma_-$ and $\Gamma_+$ corresponding to the points $(\tau,x)\in\{1/2\}\times\mathrm{fix}(S)$ which belong to $\mW^u(\Gamma_-)$, where $\mathrm{fix}(S)=\{(x,y)\mid\,y=x^2/2\}$.
The points of intersection between $\mW^u(\Gamma_-)$, $\mW^s(\Gamma_+)$ and $\{\tau=1/2\}$ correspond exactly to all points in $[0,1]\times\R^d$ whose first component is $1/2$ and that belong to heteroclinic connections between $\Gamma_-$ and $\Gamma_+$, since these are precisely the trajectories satisfying the limit condition in Theorem~\ref{teor:2main_result}(iii) and its symmetric counterpart
\[
\bigl\{
z\in(0,\delta)\times\mathbb R^d
\;\big|\;
\lim_{n\to\infty}\dist\bigl(F^{n}(z),\,\Gamma_+\bigr)=0
\bigr\}\,.
\]

The central panel of Figure~\ref{fig:ThreeFigures} depicts the intersections of
$\mW^u(\Gamma_-)$ and $\mW^s(\Gamma_+)$ with the plane $\{\tau=1/2\}$, together
with the curve $\{1/2\}\times\mathrm{fix}(S)$, for three different choices of
$\Gamma$: an invariant circle of rotation number $\w=0.195$, an invariant circle
of rotation number $\w=0.205$ and a set of $5$-periodic invariant circles
contained between the two previous curves.
The side panels depict the manifolds $\mW^u(\mK_-)$ and $\mW^s(\mK_+)$ for two
of the choices of $\Gamma$.
Finally, Corollary~\ref{cor:3heteroclinicorbits}(ii) implies that every point of
$\mW^u(\Gamma_-)$ lying in $\{1/2\}\times\mathrm{fix}(S)$ also belongs to
$\mW^s(\Gamma_+)$. Equivalently, 
\[
\mW^u(\Gamma_-)\,\cap\big(\{1/2\}\times\mathrm{fix}(S)\big)\subseteq
\mW^u(\Gamma_-)\,\cap\,\mW^s(\Gamma_+)\,\cap\,\big(\{1/2\}\times\R^d\big)\,,
\]
This inclusion need not be an equality. Indeed, the central panel of Figure~\ref{fig:ThreeFigures} shows that it is strict in the present case.
\medskip
\begin{figure}
\includegraphics[width=\textwidth]{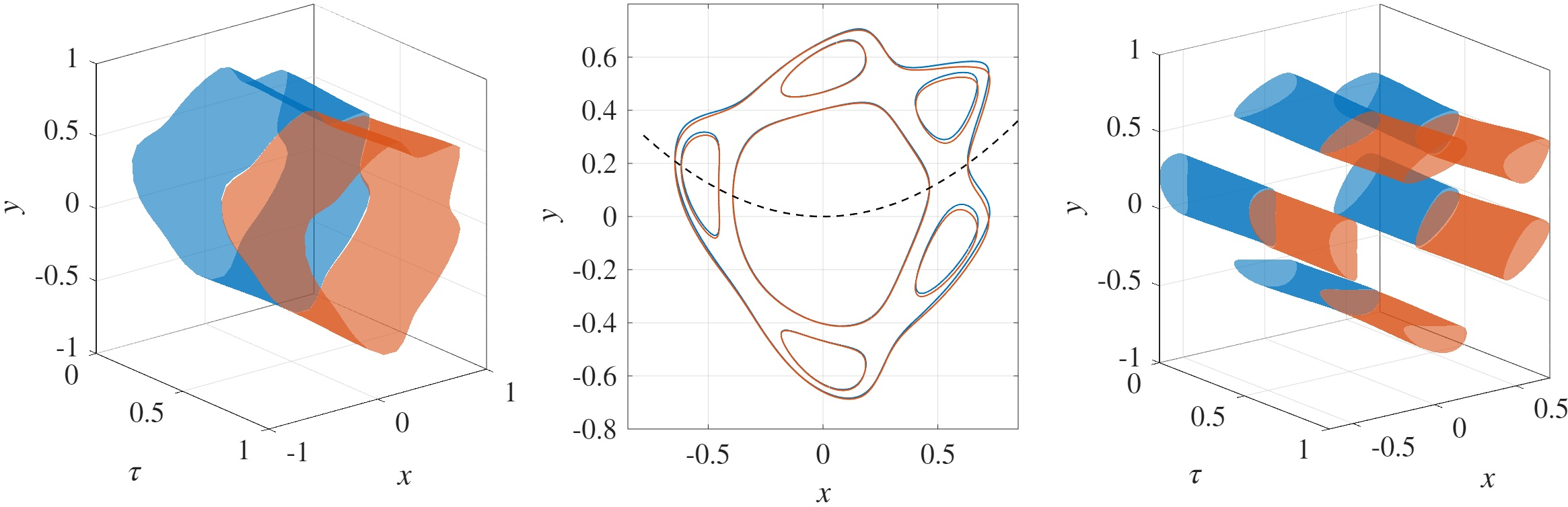}
\caption{
Stable (red) and unstable (blue) invariant manifolds associated with invariant curves of the time-dependently modulated Hénon map \eqref{eq:3_modelfirstexample}.
The left panel shows $\mW^u(\{0\}\times\Gamma_{0.195})$ 
issued from the invariant circle with rotation number $\w=0.195$, and depicted on $[0,1/2]$, together with its stable counterpart $\mW^s(\{1\}\times\Gamma_{0.195})$ on $[1/2,1]$.
The right panel shows $\mW^u(\{0\}\times\Gamma_{0.065})$ and $\mW^s(\{1\}\times\Gamma_{0.065})$
associated with an invariant curve  $\Gamma_{0.065}$  for $H_\beta^5$
rotating about a period-five elliptic point with rotation number $\w=0.065$.
The middle panel shows the intersection
with $\{\tau=1/2\}$ of the invariant manifolds displayed in the side panels
together with the intersection of the invariant manifolds of an invariant circle
rotating about the origin with rotation number
$\w=0.205$ and the reference parabola $\mathrm{fix}(S)=\{(x,y)\mid\,y=x^2/2\}$.
}
\label{fig:ThreeFigures}
\end{figure}

\textbf{Second example.}
We consider the map
\begin{equation}\label{eq:3Henonatres}
F(\tau,x,y)=\big(A(\tau),\, T_{-q(1-A(\tau))}\circ R_{\beta(q(\tau)-q(1-A(\tau))}\circ T_{q(\tau)}(x,y)\big)\,,
\end{equation}
where $A$ is of the form~\eqref{eq:2Adefinition}, $\beta=0.212$, and 
$q(\tau)=2\tau^3-3\tau^2+1$ is a cubic polynomial such that $q(0)=1$, $q(1)=0$ and $q'(0)=q'(1)=0$.
First, note that $T_a^{-1}=T_{-a}$, $R_\beta^{-1}=R_{-\beta}$ and $T_0=\mathrm{Id}$.
Then,
\[
F(0,x,y)=\big(0,\, T_0\circ R_\beta\circ T_1(x,y)\big)=\big(0,\, R_\beta\circ T_1(x,y)\big)=\big(0,H_\beta(x,y)\big)\,,
\]
and analogously $F(1,x,y)=(1,H_{\beta}^{-1}(x,y))$.
We now check that this model satisfies the hypotheses of Proposition~\ref{prop:3supersymmetry} with $\mC=[0,1]\times\R^d$.
Since it is easy to check that $F$ is invertible, it only rests to check that  
(iii) holds.
We denote $B(\tau)=1-A(\tau)$ so 
$G(\tau,x,y)=T_{-q(B(\tau))}\circ R_{\beta(q(\tau)-q(B(\tau)))}\circ T_{q(\tau)}(x,y)$ 
recall that $B^2=\mathrm{Id}$, and therefore rewrite the map $G\big(B(\tau),G(\tau,\cdot)\big)$ as
\[
T_{-q(\tau)}\circ R_{\beta(q(B(\tau))-q(\tau))}\circ T_{q(B(\tau))}\circ T_{-q(B(\tau))}\circ R_{\beta(q(\tau)-q(B(\tau)))}\circ T_{q(\tau)}\,.
\]
The  relations $T_a^{-1}=T_{-a}$ and $R_\beta^{-1}=R_{-\beta}$ ensure that
$G\big(B(\tau),G(\tau,\cdot)\big)=\mathrm{Id}$, as we wanted to see.

\begin{figure}
\includegraphics[width=\textwidth]{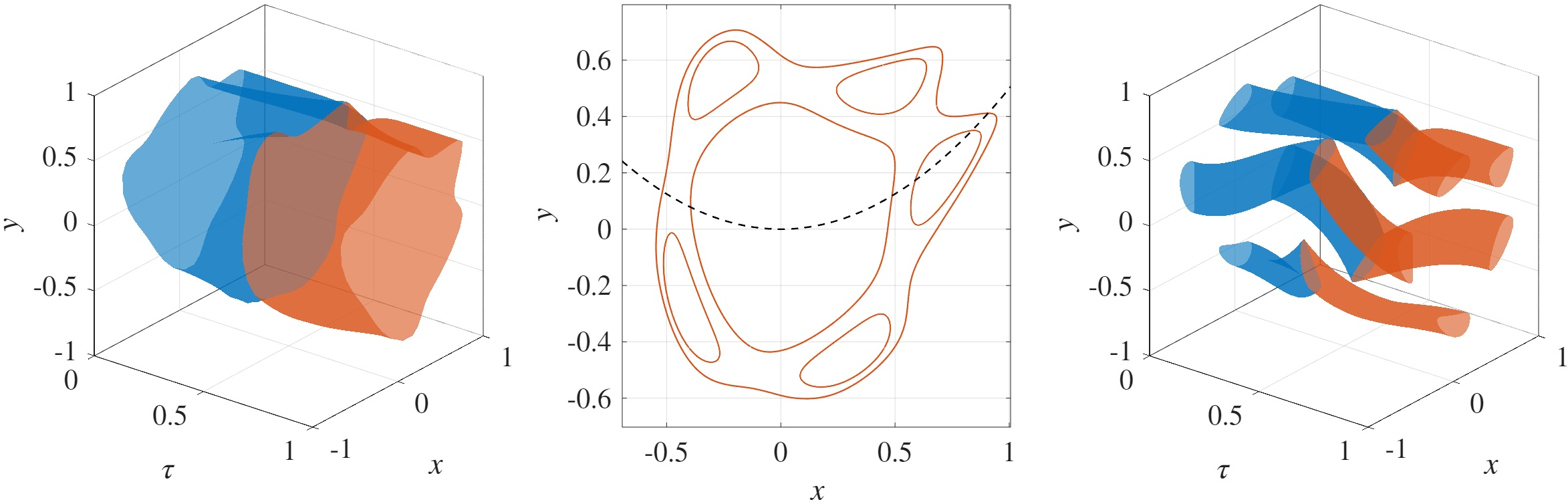}
\caption{
Stable and unstable invariant manifolds associated with invariant curves of the Hénon-related time-dependent map~\eqref{eq:3Henonatres}.
The selected curves of the past system for the representations are the same chosen in Figure~\ref{fig:ThreeFigures}.
In the side panels, the unstable manifolds are depicted over the $\tau$-interval $[0,0.5]$ whereas the stable manifolds are depicted over $[0.5,1]$.
The middle panel shows the perfect coincidence of the unstable and stable manifolds at $\tau=1/2$ of the desired objects.
}
\label{fig:ThreeFiguresV2}
\end{figure}

 As in the previous example, since $\hat\Lambda'(0)=0$, after some calculations it can be checked that $\|DF^{-1}(0,x,y)\|_2=\max\{e^{-\alpha},\sqrt{1+(x\cos\alpha+y\sin\alpha)^2}+|x\cos\alpha+y\sin\alpha|\}$.
 Moreover, we take
$\alpha=1.8$ in the map $A$ given by~\eqref{eq:2Adefinition},
 and the condition (H4) of Theorem~\ref{teor:2main_result} holds for the selected curves with $\ell=1$.

The central panel of Figure~\ref{fig:ThreeFiguresV2} depicts the intersections of $\mW^u(\Gamma_-)$ and $\mW^s(\Gamma_+)$ with the plane $\{\tau=1/2\}$, together with the curve $\{1/2\}\times\mathrm{fix}(S)$, for the same three choices of invariant curves $\Gamma$
considered in Figure~\ref{fig:ThreeFigures}.
The intersections $\mW^u(\Gamma_-)\cap\{\tau=1/2\}$ and $\mW^s(\Gamma_-)\cap\{\tau=1/2\}$ 
are indistinguishable in the figure, 
because they coincide in the considered section, in accordance with Proposition~\ref{prop:3supersymmetry}.
Conversely, as illustrated in
Figure~\ref{fig:3tau06intersection},
the manifolds $\mW^u(\Gamma_-)$ and $\mW^s(\Gamma_+)$ do not coincide on planes whose $\tau$-coordinate is not of the form $A^n(1/2)$ for some $n\in\Z$.

\begin{figure}
\includegraphics[width=\textwidth]{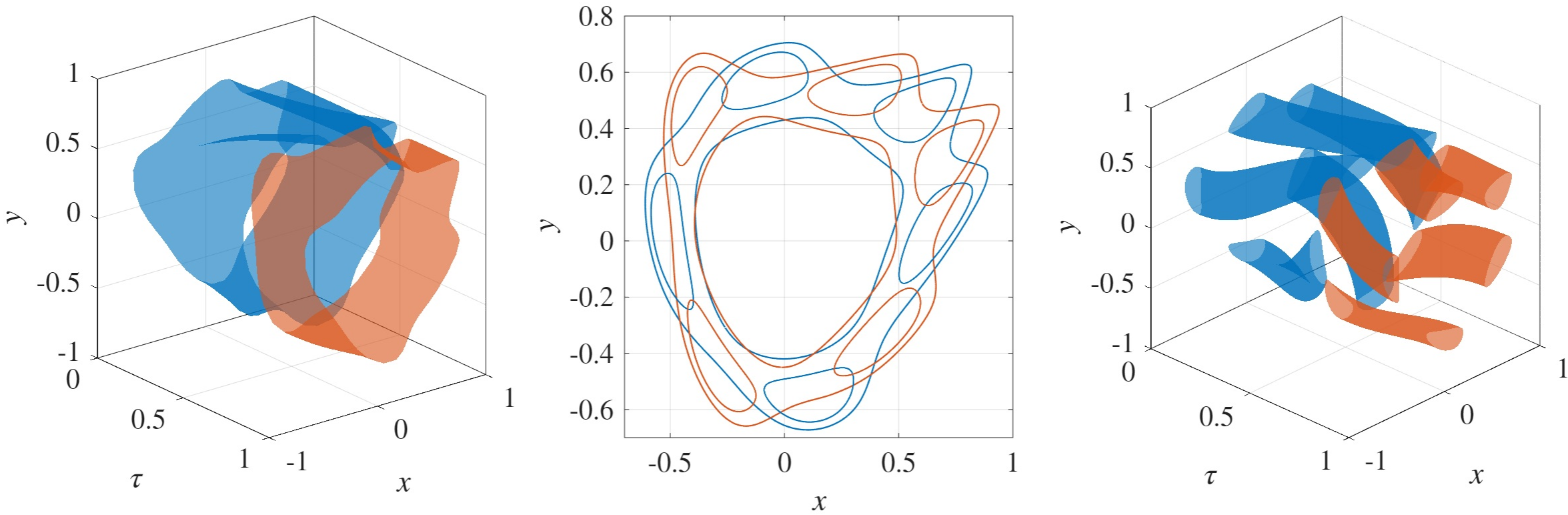}
\caption{
Stable and unstable invariant manifolds associated with invariant curves of the Hénon-related time-dependent map \eqref{eq:3Henonatres}.
The selected curves of the past system for the representations are the same chosen in Figure~\ref{fig:ThreeFigures}.
In the side panels, the unstable manifolds are depicted over the $\tau$-interval $[0,0.6]$ whereas the stable manifolds are depicted over $[0.6,1]$.
The middle panel shows the intersection of the unstable and stable manifolds at $\tau=0.6$ of the involved objects.
}
\label{fig:3tau06intersection}
\end{figure}

%THE SIMPLE EXAMPLE
\section{An example of Phase-Induced Tipping caused by Splitting}
In this section,  we consider a perturbative setting and, using a simple example, we discuss  
the role of the splitting of the invariant manifolds $\mW^u(\mK_-)$ and $\mW^s(\mK_+)$ 
in the context of tipping induced by changes of the phase.
In the unperturbed system, these manifolds coincide along a heteroclinic connection joining $\mK_-$ and $\mK_+$.
We consider a two-parameter family of perturbations depending on $(\ep,\delta)$ and 
study the splitting between the perturbed invariant manifolds.
Such analysis clarifies the role of this splitting in the onset of tipping phenomena within the framework of regular perturbation theory.

Let $0 < I_0 < I_1$, and consider the following two-parameter map 
\[
F_{\ep,\delta}\colon[0,1]\times\T^1\times\R\to[0,1]\times\T^1\times\R\,,
\]
already written in the compactified framework \eqref{eq:compactified} introduced in the previous section and given by
\begin{equation} \label{Fcomplete}
F_{\ep,\delta}(\tau,\varphi,I)
=
\big(A(\tau)\,,\,\varphi + \omega\,,\, (1+ \ep\,g_1(\tau))\, I+\delta\,g_2(\tau)\psi(\varphi)(I-I_0)+g_3(\tau) f(I)\big)\,,
\end{equation}
where $\w\in\R\setminus\mathbb{Q}$,
$f\colon\R\to\R$ satisfies $f(I_1)=0$, $\psi\colon\T^1\to\T^1$,
$g_i\colon[0,1]\to\R$ for $i\in\{1,2,3\}$ satisfy $g_2\geq0$ and
\begin{equation}\label{eq:g12conditions}
g_1(0)=g_1(1)=g_1'(0)=g_1'(1)=g_2(0)=g_2(1)=g_3(0)=0\,.
\end{equation}
In order to apply Theorem \ref{teor:2main_result},
we assume that all the involved functions are sufficiently regular to ensure that $F_{\ep,\delta}$ is a $C^{\ell+1}$ diffeomorphism.
%The  map $A$ is  given by \eqref{eq:2Adefinition} with $\alpha=1$.
It is not hard to check that $g_0(\varphi,I)=(\varphi+\omega,I)$ satisfies $\|D g_0(\varphi,I)\|_2=1$ and that
\[
DF_{\ep,\delta}(F_{\ep,\delta}^{-1}(0,\varphi,I))=
\left(
\begin{array}{ccc}
e^\alpha & 0 & 0\\
0 & 1 & 0 \\
\xi & 0 & 1
\end{array}
\right)\,,
\]
where $\xi=\delta\,g'_2(0)\psi(\varphi-\omega)(I-I_0)+g'_3(0)f(I)$. 
Let the map $A$ be given by \eqref{eq:2Adefinition} with $\alpha=1$. Then,
it can be 
shown that  $\|DF_{\ep,\delta}(F_{\ep,\delta}^{-1}(0,\varphi,I))\|_2 \leq 2$
whenever $|\xi|\leq 3/2$, which will be the situation in the examples below.
Thus, the gap condition (H4) of Theorem \ref{teor:2main_result} holds.

Note that, given $\mK:=\{(\varphi,I_1)\mid\,\varphi\in\T^1\}$, the tori
$\mK_-:=\{0\}\times\mK$ and $\mK_+:=\{1\}\times\mK$ are invariant under
$F_{\ep,\delta}$. We assume that $\mK_-$ satisfies the hypotheses of Theorem
\ref{teor:2main_result} and the hypotheses of an analogous time-reversed
theorem, so that there exist the invariant manifolds $\mW^u(\mK_-)$ and
$\mW^s(\mK_+)$.

Variation in the parameter $\ep$ is responsible (after a proper choice of $g_1$)
for the transverse oscillation associated with the splitting of the heteroclinic
orbit. The parameter $\delta$, in turn, introduces a dissipative effect towards
$I_0$ (in the $\ep$-unperturbed case) that decays asymptotically as $\tau$
approaches $0$ or $1$, since $g_2(0)=g_2(1)=0$, thereby causing the stable and
untable manifolds to separate as $\delta$ increases.
The function $\psi$ allows for an angular modulation of the dissipation and will allow us to produce some interesting examples.
Finally, the nonlinear term $f(I)$ allows several different dynamical regimes to be treated within a unified framework: asymptotically conservative behaviour in both the past and the future, asymptotically conservative behaviour in the past but asymptotically dissipative behaviour in the future, and so forth.
Notice that variations in $\ep$ and $\delta$ do not affect the dynamics on the past and future planes $\{\tau=0\}$ and $\{\tau=1\}$.

For the unperturbed system with $\ep=0$ and $\delta=0$, the manifold
\begin{equation}\label{eq:3unperturbed_heteroclinic}
[0,1]\times\T^1\times\{I_1\}
\end{equation}
is invariant under $F_{0,0}$ and is composed of heteroclinic connections joining the past invariant torus $\mK_-$ to the future invariant torus $\mK_+$.
This is the two-dimensional heteroclinic manifold $\mW^u(\mK_-)=\mW^s(\mK_+)$ whose
splitting we shall study as a function of the parameters $\ep$ and $\delta$.

The Melnikov function  generally  provides a  systematic way  for analyzing the first-order splitting  that leads to to phase-change induced tipping.
However, the examples provided in the next section involve perturbations simple enough that the splitting function can be computed easily, either numerically or analytically.

\subsection{Numerical evidence and results.}\label{sec:41numericalevidenceandresults}
In the simpler case in which $\psi\equiv1$ in~\eqref{Fcomplete}, the independence of the evolution in the variable $I$ with respect to the angle $\varphi$ allows to simplify the study to that of the reduced two-dimensional map
\begin{equation} \label{2d_namap}
\hat F_{\ep,\delta}(\tau,I)=\big(A(\tau)\,,\; (1+\ep\,g_1(\tau))\, I+\delta\,g_2(\tau)(I-I_0)+g_3(\tau)\, f(I)\big)\,,
\end{equation}
removing also the $\varphi$-dependence from the splitting function.
Throughout this section, we use this reduced planar model to illustrate
and discuss the numerical results. 
Using the  previously introduced diffeomorphism $\Theta(t)=1/(1+e^{-t})$ (that is, we consider $\alpha=1$ in~\eqref{eq:2Thetadefinition}), we compute the invariant manifolds of the dynamical system on $[0,1]\times\mathbb{R}$ and relate the relative splitting oscillations of these manifolds to different phase-tipping properties. 
In the examples below, the splitting function is computed by approximating the stable and unstable manifolds through the parameterization method in a fundamental domain and then globalized.

Figure \ref{fig:epsilondeltasplitting} illustrates the role of the parameters $\ep$ and $\delta$ in the splitting between $\mW^u(0,I_1)$ and $\mW^s(1,I_1)$ for \eqref{2d_namap}, for a fixed value of $I_1$.
As established in~\eqref{eq:3unperturbed_heteroclinic}, these invariant manifolds coincide
when $\ep=\delta=0$.
To simplify the setting considered here, we additionally take $g_3\equiv0$.
For the choice of $g_1(\tau)= \exp(-(\Theta^{-1}(\tau))^2)\,\sin(2\pi\Theta^{-1}(\tau))$ made
in this example, varying $\ep$ produces oscillations
between the invariant manifolds of the tori of the limiting systems.
In contrast, taking $g_2(I)=\tau\,(1-\tau)$,
the parameter $\delta$, which can be interpreted as dissipation towards $I_0$, produces a relative displacement between these invariant  unstable and stable manifolds. 
In particular, for $\delta$ large enough, the invariant manifolds do not intersect.
This explains the reduction in the size of the 
region where the splitting function is negative
as the value of $\delta$ is increased.

\begin{figure}
\includegraphics[width=\textwidth]{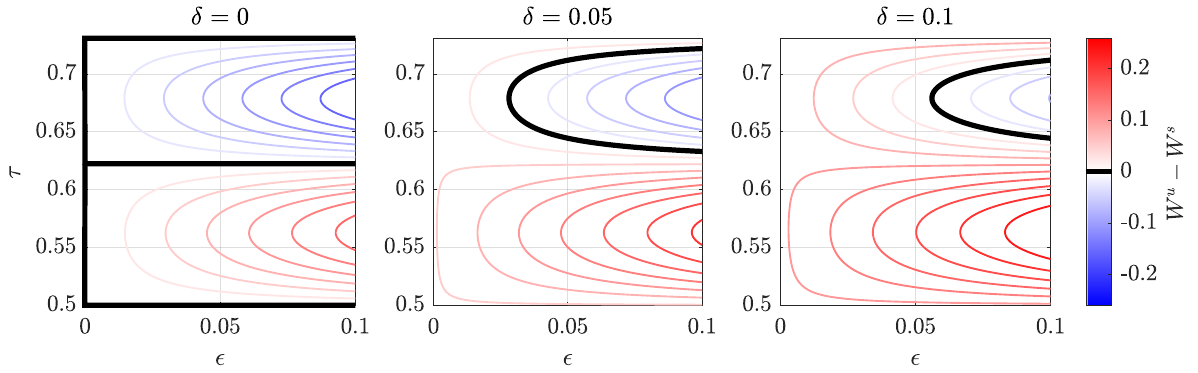}
\caption{
Level curves of the splitting function $W^{u}_{\mK_-}(\tau)-W^{s}_{\mK_+}(\tau)$ on the strip $\tau\in[\tau_0,A(\tau_0)]$ with $\tau_0=1/2$, shown as a function of $\ep$ (horizontal axis) and $\tau$ (vertical axis) for several fixed values of $\delta$.
In this example, we take $I_0=0$, $I_1=1$, $\psi\equiv1$, and $g_{3}\equiv 0$, while the fiber dynamics is defined using
$g_1(\tau)=
\exp(-(\Theta^{-1}(\tau))^2)\,\sin(2\pi\Theta^{-1}(\tau))$
and
$g_{2}(\tau)=\tau\,(1-\tau)$.
Colors encode both the sign and the magnitude of the difference between the
unstable and stable manifolds: red corresponds to positive values, blue to
negative values, and in black the zero level curve.
The curve levels are equally spaced between the maximum and the minimum reached
within the three panels.
}
\label{fig:epsilondeltasplitting}
\end{figure}

Figure \ref{fig:entangledsplitting} shows the entanglement of $\mW^u(0,I_1)$ and $\mW^s(1,I_1)$ 
varying $I_1$, for different values of $\ep$, with $\delta=0$ fixed and $g_3\equiv0$.
So, the term involving $g_2$ plays no role.
Here, $g_1$ is chosen as in Figure~\ref{fig:epsilondeltasplitting}.
If we add the the third variable $\varphi$  back
to the representation of Figure \ref{fig:entangledsplitting}, keeping
$\psi\equiv1$, we would observe that the two-dimensional manifolds
$\mW^u(\mK_-)$ and $\mW^s(\mK_+)$, where $\mK_-:=\{0\}\times\T^1\times\{I_1\}$
and $\mK_+:=\{1\}\times\T^1\times\{I_1\}$,
intersect on a sequence of circles for each choice of $I_1$.
Each point on one of these circles is a representative of a different
heteroclinic orbit within the fundamental domain, hence,
we obtain a one-dimensional continuum of heteroclinic orbits connecting $\mK_-$
to $\mK_+$.

\begin{figure}
\includegraphics[width=\textwidth]{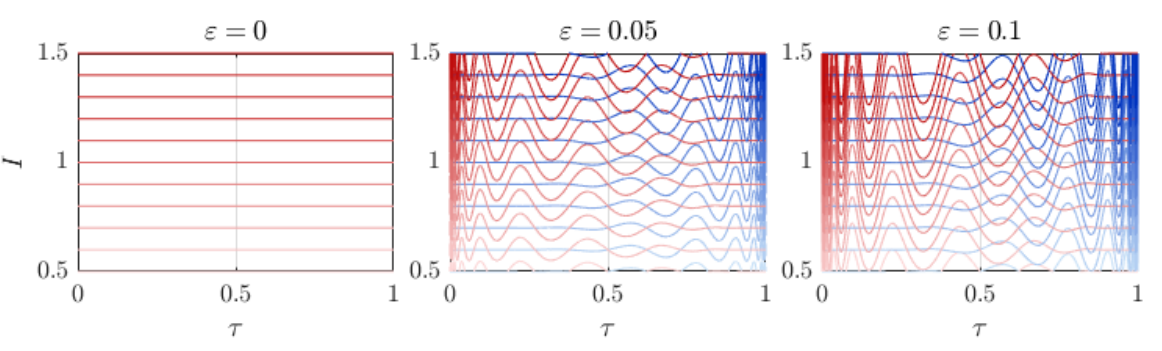}
\caption{Unstable and stable invariant manifolds for $F_{\ep,\delta}$ in the
$(\tau,I)$-plane for a family of fixed points $(0,I_1)$ and $(1,I_1)$,
respectively, with $I_1$ varying over a prescribed range.
Each panel corresponds to a different value of $\ep$, while $\delta=0$ is kept fixed.
The function $g_1$ has been chosen as in Figure \ref{fig:epsilondeltasplitting}.
The families of invariant manifolds $\mathcal{W}^u(0,I_1)$ is plotted in blue
and $\mathcal{W}^s(1,I_1)$ in red; different shades distinguish individual
members.
}
\label{fig:entangledsplitting}
\end{figure}

\medskip

Going back to the planar map, in Figure~\ref{fig:edampedsplitting}, we take $\ep=\delta=0$,
so that the terms involving $g_1$ and $g_2$ play no role in \eqref{2d_namap}.
We consider 
$$g_3(\tau)= \exp(-(\Theta^{-1}(\tau))^2)\,\sin(2\pi\Theta^{-1}(\tau)) \quad \text{and} \quad f(I)=\eta\,(I-0.5)(I-1.5),$$ 
to  illustrate the values of $I$ that are reachable from the
unstable manifold  $\mW^u(0,I)$ parametrized by $\tau\mapsto(\tau,W_I^u(\tau))$.
That is, we consider the  reduced two-dimensional map 
\[
\hat F_{0,0}(\tau,I)=\big(A(\tau),\,I+g_3(\tau)\,f(I)\big)\,.
\]
The interval of critical points of the future equation that are accessible through orbits on the unstable manifold of $(0,I)$ has length
$$\limsup_{\bar\tau\to1^-}\left(\max_{\tau\in[\bar\tau,1]}W_I^u(\tau)-\min_{\tau\in[\bar\tau,1]}W_I^u(\tau) \right).$$
Note that, since $f(I)<0$ for $I\in(0.5,1.5)$, the invariance equation $W^u_I(A(\tau))=W_I^u(\tau)+g_3(\tau)f(W_I^u(\tau))$ shows that $W_I^u(A(\tau))<W_I^u(\tau)$ whenever $g_3(\tau)>0$ and $W_I^u(A(\tau))>W_I^u(\tau)$ whenever $g_3(\tau)<0$.
Since it can be deduced from \eqref{eq:torecall2definitionLambdaHat} and \eqref{eq:torecall3definitionA} that $\sin(2\pi\Theta^{-1}(A^n(\tau)))=\sin(2\pi\Theta^{-1}(\tau))$ for all $\tau\in[0,1]$ and $n\in\Z$, we get that $(g_3(A^n(\tau)))_{n\in\Z}$ has constant sign for every $\tau\in(0,1)$.
So, it follows that the length of the interval of accesible orbits coincides with
$\max_{\tau\in[0,1]}W_I^u(\tau)-\min_{\tau\in[0,1]}W_I^u(\tau)$.
This quantity measures the 
size of the set of equilibrium points of the future equation 
asymptotically connected to the point $(0,I)$ by a heteroclinic orbit.
The maximum spread of the heteroclinic orbits is attained at $I = 0.5$.
The overall described situation is consistent with
phase-induced tipping in the sense that
nearby points on the invariant manifold, selected by different phases,
may have different $\upomega$-limit sets. Nevertheless, this cannot
properly be regarded as a critical transition, since the
$\upomega$-limit set does not change abruptly as a function of the phase.

\begin{figure}
\includegraphics[width=0.75\textwidth]{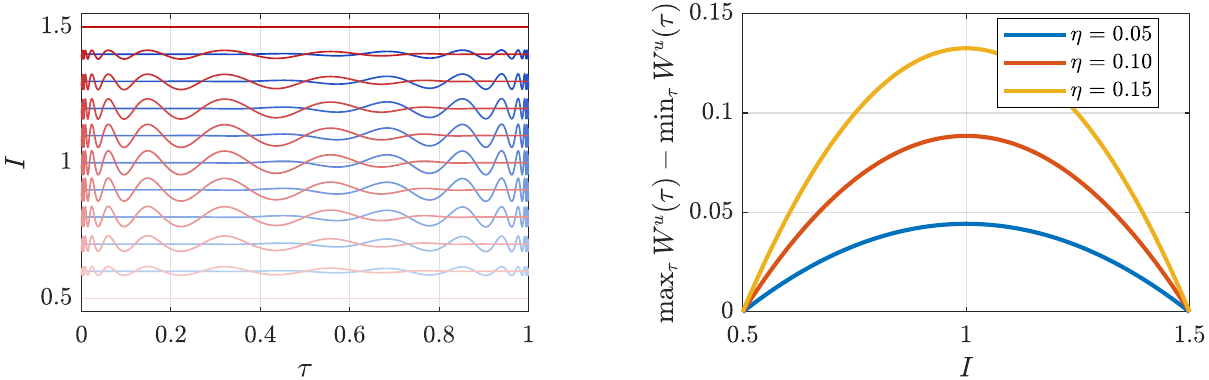}
\caption{In the left panel, unstable and stable invariant manifolds for $F_{\ep,\delta}$ in the $(\tau,I)$-plane for a family of fixed points $(0,I_1)$ and $(1,I_1)$, respectively, computed with $\ep=0$, $\delta=0$, $f(I)=\eta\,(I-0.5)(I-1.5)$, $\eta=0.1$, $g_3(\tau)=
\exp(-(\Theta^{-1}(\tau))^2)\,\sin(2\pi\Theta^{-1}(\tau))$, and a range of values of $I_1$.
In the right panel, the range of $W^u_I(\tau)$ 
is shown for different values of $\eta$.
}
\label{fig:edampedsplitting}
\end{figure}

\medskip

Finally, Figure~\ref{fig:ecriticaltransitionsplitting}  displays results corresponding to~\eqref{2d_namap} with
\begin{align*}
&g_1(\tau)=\exp(-(\Theta^{-1}(\tau))^2)\sin(2\pi\Theta^{-1}(\tau)), 
&f(I)=\frac{2}{\pi}\arctan(2(I-1))+1-I,
\end{align*}
and $g_3(\tau)=\tau$.
We set $\delta=0$, so that the term involving $g_2$ plays no role.
It is not hard to check that the future system, which is dissipative, has three
fixed points, located at $I_1=1/2$, $I_1=1$, and $I_1=3/2$.
The fixed point at $I_1=1$ is repelling, whereas those at $I_1=1/2$ and $I_1=3/2$ are attracting.
Thus, with these choices, system~\eqref{2d_namap} is asymptotically dissipative as $t\to\infty$ and asymptotically conservative as $t\to-\infty$.
Initially, for $\ep = 0$, the unstable manifold of $(0,1)$ and the stable manifold of $(1,1)$ coincide.
When $\ep$ becomes positive they split in a similar fashion to that previously discussed in Figures \ref{fig:epsilondeltasplitting}, \ref{fig:entangledsplitting} and \ref{fig:edampedsplitting}.
Nevertheless, a key difference of the case shown in
Figure \ref{fig:ecriticaltransitionsplitting} is that
the only possible $\upomega$-limit sets are $(1,1/2)$, $(1,1)$ and $(3/2,1)$.
As seen in the figure, every trajectory contained in the unstable manifold
$\mW^u(0,1)$ for which $W^u-W^s>0$ occurs converges to $(3/2,1)$,
as happens, in particular, for the orbit formed by the black dots shown in the figure.
Likewise, every trajectory for which case $W^u-W^s<0$ occurs converges to $(1,1/2)$, including the green orbit depicted.
In this context, the phase-induced tipping observed here can be characterized as a genuine critical transition: once $\ep$ is fixed, arbitrarily small variations in the phase---that is, small changes in the choice of trajectory along the unstable manifold---lead to sharply distinct $\upomega$-limit sets.

\begin{figure}
\includegraphics[width=\textwidth]{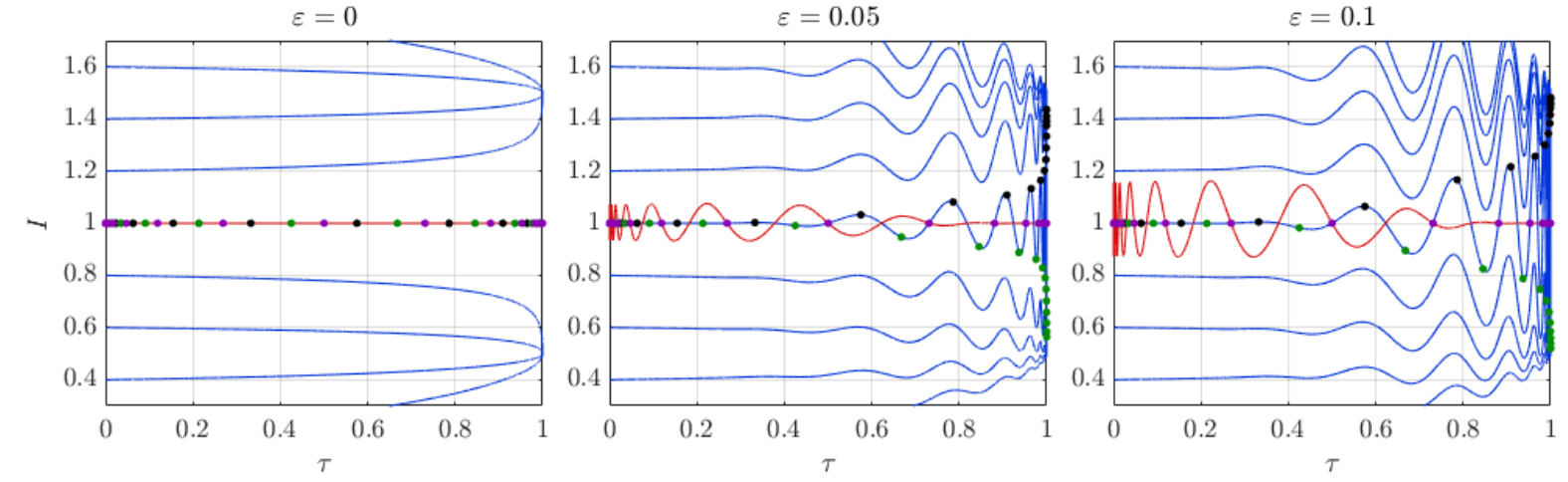}

\caption{We consider $F_{\ep,\delta}$ with $\delta=0$, $I_0=1$, $g_1(\tau)=\exp(-(\Theta^{-1}(\tau))^2)\sin(2\pi\Theta^{-1}(\tau))$, $g_3(\tau)=\tau$, $f(I)=(2I_0/\pi)\arctan(2(I/I_0-1))+I_0-I$.
For different values of $\ep$, we display in blue $\mathcal{W}^u(0,I_1)$  for a range of values of $I_1$, together 
with $\mathcal{W}^s(1,1)$ in red.
The color points correspond to orbits on $(\tau,W^u_1(\tau))$ obtained from the sections $\tau_n=A^n(\tau_0)$ with $\tau_0=0.425$ (greeen), $\tau_0=0.5$ (purple), and $\tau_0=0.575$ (black), respectively.}
\label{fig:ecriticaltransitionsplitting}
\end{figure}

\subsection{Surgical intersections between manifolds.}\label{sec:42surgicalintersections}
In this section we first provide
an example such that
the intersection of the two-dimensional manifolds $\mW^u_{\ep,\delta}(\mK_-)$ and $\mW^s_{\ep,\delta}(\mK_+)$  reduces to a single orbit
$(A^n(\tau_*),x_n)_{n \in \mathbb{Z}}$
for a given $\tau_* \in (0,1)$.
After that, we modify the example so that the corresponding invariant manifolds
intersect in a Cantor set of orbits.
These examples will be relevant in underscoring the connection between the two different approaches to phase-induced tipping given in Section \ref{subsec:phaseinducedtipping}.
 For the sake of clarity, throughout this section we explicitly include the subscripts $\ep$ and $\delta$ in the notation for the stable and unstable invariant manifolds.

\subsubsection{A single orbit as the intersection of 2D invariant manifolds} \label{1orb}
To construct such an example,
let us take $g_3\equiv0$
in~\eqref{Fcomplete}.
That is, consider on $[0,1]\times\T^1\times\R$ the map
\[
F_{\ep,\delta}(\tau,\varphi,I)=\big(A(\tau),\,\varphi+\omega,\;(1+\ep\,g_1(\tau))\,I+\delta\,g_2(\tau)\,\psi(\varphi)\,(I-I_0)\big)\,.
\]
Starting from the unperturbed system $\ep=\delta=0$, we construct such an example in two steps.
First,
we choose a function $g_1$ and fix a sufficiently small $\ep>0$, thereby obtaining a splitting such as that shown in Figure~\ref{fig:entangledsplitting}.
Second, we choose suitable functions $g_2\geq0$ and $\psi\geq0$ so that the resulting perturbed system, with 
the previously fixed value of $\ep$ 
and $\delta=1$, has the desired properties.

We consider the same choice of $g_1(\tau)=
\exp(-(\Theta^{-1}(\tau))^2)\,\sin(2\pi\Theta^{-1}(\tau))$ as in Figure \ref{fig:entangledsplitting}, where $\Theta(t)=1/(1+e^{-t})$.
For $\delta=0$, thereby considering a $\varphi$-independent parametric variation, we fix a small value of $\ep>0$.
Let $(\tau,\varphi)\mapsto(\tau,\varphi,\hat W_\ep^
u(\tau))$ with $\hat W_\ep^u\colon[0,1)\to\R$ be a parametrization of $\mW_{\ep,0}^u(\mK_-)$, and let $\hat W_\ep^s$ define the analogous parametrization of $\mW_{\ep,0}^s(\mK_+)$, 
defined on $(0,1]$. As before, given $I_1 \in \R$, we denote $\mK_-=\{0\}\times\T^1\times\{I_1\}$ and $\mK_+=\{1\}\times\T^1\times\{I_1\}$.
Both $\hat W_\ep^u$ and $\hat W_\ep^s$ satisfy the invariance equation
\begin{equation}\label{eq:4epsilon_invariance_eq}
\hat W_\ep(A(\tau))=(1+\ep\,g_1(\tau))\,\hat W_\ep(\tau)\,,
\end{equation}
and, moreover, $\hat W_\ep^u(0)=\hat W_\ep^s(1)=I_1$.
Note that the choice of $g_1$ guarantees that $g_1(A^n(1/2))=0$ for every $n\in\mathbb{Z}$. Substituting this identity into \eqref{eq:4epsilon_invariance_eq}, we find that both $\hat W_\ep^u$ and $\hat W_\ep^s$ are constant along the orbit of 
the point $(1/2,I_1)$; that is, one has
$\hat{W}_\ep(A^n(1/2))=\hat W_\ep(A^m(1/2))$ for all $n,m\in\mathbb{Z}$.
Letting $m\to\infty$ and using the asymptotic behavior of the corresponding parametrizations, we conclude that $\hat W_\ep^u(A^n(1/2))=\hat W_\ep^s(A^n(1/2))=I_1$ for every $n\in\mathbb{Z}$.

Now, we look for appropriate functions $g_2$ and $\psi$ such that $\mW^u_{\ep,\delta}(\mK_-)\cap\mW^s_{\ep,\delta}(\mK_+)$ reduces to a single orbit when $\delta=1$.
Figure~\ref{fig:explicacion_atrajectory} displays the functions involved in the construction; referring to it may facilitate the reading of the text.
First, we impose a simplifying condition on the function $g_2$, which will allow us to compute the splitting function explicitly on a prescribed interval for the perturbation in $\ep$ and $\delta$.
We fix $\bar\tau\in(0,1)$ such that $\hat W^u_\ep(\bar\tau)>\hat W^s_\ep(\bar\tau)$,
in particular, $\bar\tau\not\in\{A^n(1/2)\}_n$.
We choose the function $g_2$ in such a way that $g_2(\tau)>0$ for $\tau\in(\bar\tau,A(\bar\tau))$ and $g_2(\tau)=0$ otherwise.
Notice that, under these assumptions,
\[
F_{\ep,\delta}(\tau,\varphi,I)=\big(A(\tau),\,\varphi+\w,\,(1+\ep\,g_1(\tau))\,I\big)=F_{\ep,0}(\tau,\varphi,I)
\]
for $\tau\not\in(\bar\tau,A(\bar\tau))$.
Hence, if the unstable invariant manifold $\mW^u_{\ep,\delta}(\mK_-)$ is parametrized by $(\tau,\varphi)\mapsto (\tau,\varphi,W_{\ep,\delta}^u(\tau,\varphi))$, the map $W_{\ep,\delta}^u$ satisfies the invariance equation \eqref{eq:4epsilon_invariance_eq} for $\tau\leq\bar\tau$.
This, together with the asymptotic behaviour of $W_{\ep,\delta}^u(\tau,\varphi)\to I_1$ as $\tau\to0$ 
ensures that $W^u_{\ep,\delta}(\tau,\varphi)=\hat W^u_\ep(\tau)$ for all $\tau\leq A(\bar\tau)$.
Analogously, if the stable invariant manifold $\mW^s_{\ep,\delta}(\mK_+)$ is parametrized by $(\tau,\varphi)\mapsto(\tau,\varphi,W^s_{\ep,\delta}(\tau,\varphi))$, then we obtain that $W^s_{\ep,\delta}(\tau,\varphi)=\hat W^s_\ep(\tau)$ for all $\tau\geq A(\bar\tau)$.
Therefore, for $\tau\in[\bar\tau,A(\bar\tau)]$ we have that
\[
\begin{split}
W^s_{\ep,\delta}(A(\tau),\varphi+\omega)
&=\hat W^s_\ep(A(\tau))=(1+\ep\,g_1(\tau))\,\hat W^s_\ep(\tau)\\
&=\big(1+\ep\,g_1(\tau)+\delta\,g_2(\tau)\,\psi(\varphi)\big)\,W^s_{\ep,\delta}(\tau,\varphi)-\delta\,I_0\,g_2(\tau)\,\psi(\varphi)\,,
\end{split}
\]
from where we get that
\[
W^s_{\ep,\delta}(\tau,\varphi)
=\frac{(1+\ep\,g_1(\tau))\,\hat W^s_\ep(\tau)+\delta\,I_0\,g_2(\tau)\,\psi(\varphi)}{1+\ep\,g_1(\tau)+\delta\, g_2(\tau)\,\psi(\varphi)}\,,\quad\text{for }\tau\in[\bar\tau,A(\bar\tau)]\,.
\]
Therefore,
we can explicitly write down the splitting function for $\tau\in[\bar\tau,A(\bar\tau)]$ as
\begin{equation}
    \label{eq:4splittingsurgical}
S_{\ep,\delta}(\tau,\varphi)=\frac{(1+\ep\,g_1(\tau))(\hat W^u_\ep(\tau)-\hat W_\ep^s(\tau))+\delta\,g_2(\tau)\,\psi(\varphi)(\hat W_\ep^u(\tau)-I_0)}{1+\ep\,g_1(\tau)+\delta\,g_2(\tau)\,\psi(\varphi)}\,.
\end{equation}
And thus, for $(\tau,\varphi)\in[\bar\tau,A(\bar\tau))\times\T^1$, we get that $S_{\ep,\delta}(\tau,\varphi)=0$ if and only if
\[
(1+\ep\,g_1(\tau))(\hat W^u_\ep(\tau)-\hat W_\ep^s(\tau))+\delta\,g_2(\tau)\,\psi(\varphi)(\hat W_\ep^u(\tau)-I_0)=0\,,
\]
and in fact $S_{\ep,\delta}$ and $(1+\ep\,g_1)(\hat W^u_\ep-\hat W_\ep^s)+\delta\,g_2\,\psi(\hat W_\ep^u-I_0)$ have the same sign at every point if $\ep$ is sufficiently small. 

Recall that 
we seek the desired tangency to take place for $\delta=1$.
Hence,
we now choose the functions $g_2$ and $\psi$ so that $S_{\ep,1}$
vanishes at exactly one point in $[\bar\tau,A(\bar\tau))\times\T^1$.
We take
\[
g_2(\tau)=\eta_0(\tau)\,\frac{1+\ep\, g_1(\tau)}{\hat W_\ep^u(\tau)-I_0}\,,
\]
for some nonnegative function  $\eta_0\colon[0,1]\to[0,\infty)$ supported on $[\bar\tau,A(\bar\tau)]$. 
Thus, 
$S_{\ep,1}$ has the same sign as
\begin{equation}\label{eq:42splitting_function_vanishes}
\hat W_\ep^u(\tau)-\hat W_\ep^s(\tau)+\eta_0(\tau)\,\psi(\varphi)\,.
\end{equation}

Consider the compact set $K=\{\tau\in[\bar\tau,A(\bar\tau)]\mid\, \hat W^u_\ep(\tau)-\hat W^s_\ep(\tau)\leq 0\}$, and assume that $K\subset(\bar\tau,A(\bar\tau))$.
Note that this last assumption entails no loss of generality, since $\bar\tau$ can be chosen again, if necessary, so that it is satisfied.
Then, we select $\tau_*\in K$ with $\hat W^u_\ep(\tau^*)-\hat W^s_\ep(\tau^*)<0$ and define
\[
q(\tau):=(\tau-\tau_*)^2\,.
\]
Next,
consider the function $p(\tau)=q(\tau)-(\hat W^u_\ep(\tau)-\hat W^s_\ep(\tau))$ and note that $p(\tau)>0$ for $\tau\in K$.
Then, take an open neighbourhood $K\subset U\subset (\bar\tau,A(\bar\tau))$ such that $p(\tau)>0$ for all $\tau\in U$.
Finally, consider a $C^\infty$ function $\eta_1$
taking values in $[0,1]$ such that $\eta_1(\tau)=1$ for all $\tau\in K$ and $\eta_1(\tau)=0$ for all $\tau\not\in U$, and define
\[
\eta_0(\tau):=\eta_1(\tau)\, p(\tau)\,,
\]
\begin{figure}
\includegraphics[width=\textwidth]{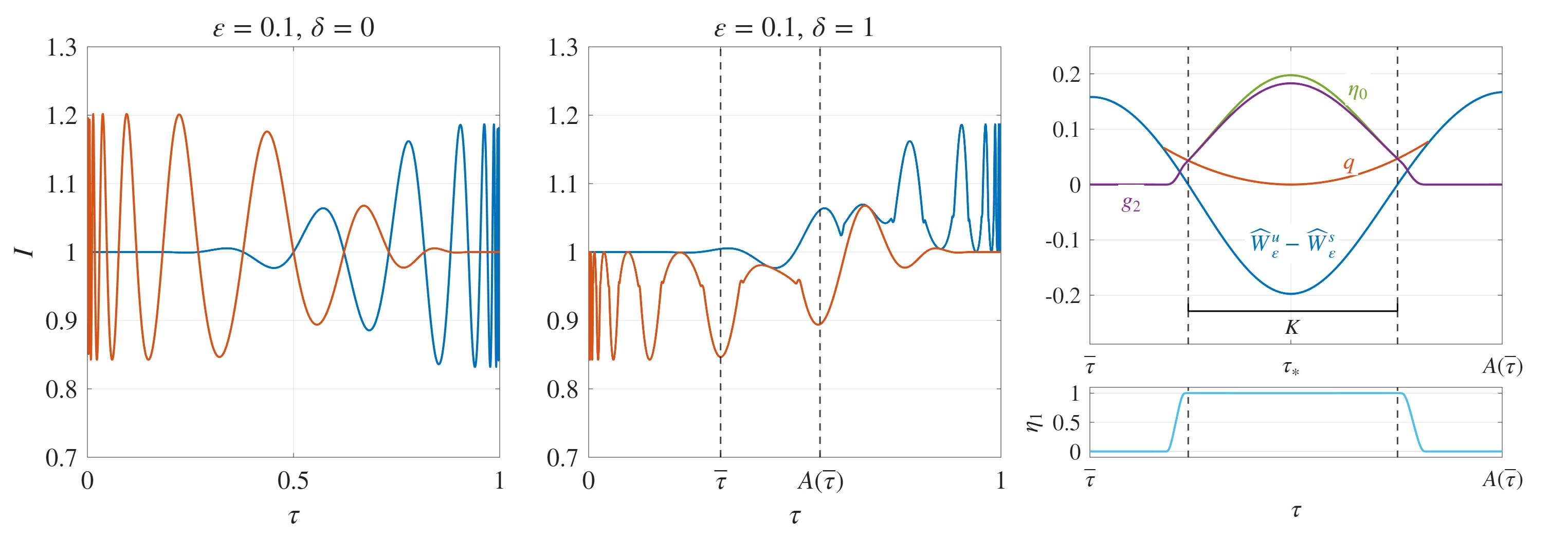}
\caption{ 
Functions involved in the construction of Section~\ref{1orb}, with $\psi(\varphi)\equiv1$ and $\varepsilon=0.1$.
The left panel shows $\mW^u(0,1)$ and $\mW^s(1,1)$ for the map with $\delta=0$.
The central panel displays the tangency between the
invariant manifolds  taking place at $\delta=1$ 
when adding the perturbation $g_2$
with $\bar\tau=0.32$, 
$I_0=0$, $\tau_*=0.43761$, $q(\tau)=12\,(\tau-\tau_*)^2$, $K\approx[0.37760, 0.49973]$ and $U=(0.36299, 0.51835)$.
The upper-right panel shows, over the interval $[\bar{\tau},A(\bar{\tau})]$, the $\delta$-unperturbed splitting function, together with the functions $q$, $\eta_0$, and $g_2$ corresponding to the previous choices.
The vertical dashed lines mark the zeros of $\hat W_\ep^u-\hat W_\ep^s$, which delimit the interval
$K$. The lower-right panel depicts the cutoff function $\eta_1$, whose plateau coincides with the interval $K$.
}
\label{fig:explicacion_atrajectory}
\end{figure}
Note that $\eta_0(\tau)\geq0$ for all $\tau \in [0,1]$.
With these choices, since
\begin{equation}\label{eq:42auxiliar}
\hat W^u_\ep(\tau)-\hat W^s_\ep(\tau)+\eta_0(\tau)=(1-\eta_1(\tau))(\hat W^u_\ep(\tau)-\hat W^s_\ep(\tau))+\eta_1(\tau)(\tau-\tau_*)^2\,,
\end{equation}
it is not hard to check that $\hat W_\ep^u(\tau)-\hat W_\ep^s(\tau)+\eta_0(\tau)>0$ for all $\tau\in[\bar\tau,\tau_*)\cup(\tau_*,A(\bar\tau)]$ and that $\hat W_\ep^u(\tau_*)-\hat W_\ep^s(\tau_*)+\eta_0(\tau_*)=0$: for $\tau\in K\setminus\{\tau_*\}$ the first term in the right hand side of \eqref{eq:42auxiliar} vanishes and the second one is strictly positive, for $\tau\in[\bar\tau,A(\bar\tau)]\setminus K$ the first term is nonnegative and the second one is strictly positive, and for $\tau=\tau_*$ both terms vanish.

In the particular setting illustrated in Figure~\ref{fig:explicacion_atrajectory}, $K$ is a closed interval and $U$ is an open interval containing it, and we construct $\eta_1$ as follows. Let $\xi\colon\mathbb{R}\to\mathbb{R}$ be defined by $\xi(s)=0$ for $s\in(-\infty,0]$, $\xi(s)=e^{-1/s}/\bigl(e^{-1/s}+e^{-1/(1-s)}\bigr)$ for $s\in(0,1)$, and $\xi(s)=1$ for $s\in[1,\infty)$.
Let $T_1$ be the affine transformation mapping $\inf U$ to $0$ and $\inf K$ to $1$, and let $T_2$ be the affine transformation mapping $\sup K$ to $1$ and $\sup U$ to $0$.
We then define $\eta_1$ by $\eta_1(\tau)=\xi(T_1(\tau))\xi(T_2(\tau))$. By construction, $\eta_1$ belongs to $C^\infty$, is identically equal to $1$ on $K$, and vanishes outside $U$.

Now, for a fixed $\varphi_*\in\T^1$, we consider
\begin{equation}
\label{eq:4psistandard}
\psi(\varphi)=2-\cos(2\pi(\varphi-\varphi_*))\,,
\end{equation}
 so that, from \eqref{eq:42splitting_function_vanishes} and using the fact that $\eta_0\geq0$, one obtains that
\[
S_{\ep,1}(\tau,\varphi)>0\quad\text{for all }(\tau,\varphi)\in[\bar\tau,A(\bar\tau)]\times\T^1\setminus\{(\tau_*,\varphi_*)\}\quad\text{and}\quad S_{\ep,1}(\tau_*,\varphi_*)=0\,.
\]
That is, the orbit of $(\tau_*,\varphi_*,W^u_{\ep,1}(\tau_*,\varphi_*))$ is the unique orbit which belongs to $\mW^u(\mK_-)\cap\mW^s(\mK_+)$, as desired.
In Figure \ref{fig:4levelcurves_surgical1}, the level curves of the splitting map $S_{\ep,\delta}$ are represented for this example.
In particular, for $\delta<1$ but close to it, we see that the
intersection between the nodal lines of $\mW^u(\mK_-)$ and $\mW^s(\mK_+)$ in the fundamental domain $[\bar\tau,A(\bar\tau)]\times\T^1$ is of bubble type~\cite{LomeliRamirez2008}.
\begin{figure}
\includegraphics[width=\textwidth]{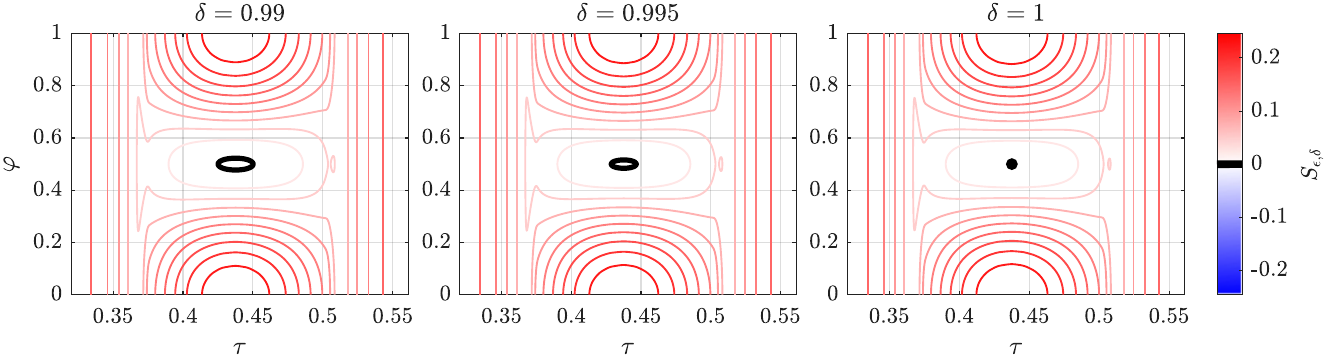}
\caption{Level curves of $S_{\ep,\delta}(\tau,\varphi)$, given in \eqref{eq:4splittingsurgical}.
 With the exception of $\delta$, which takes different values on each panel, all other functions and parameters are chosen as in Figure~\ref{fig:explicacion_atrajectory}.
The thick black curve, which reduces to a point in the right panel, denotes the zero level $S_{\ep,\delta}=0$.
}
\label{fig:4levelcurves_surgical1}
\end{figure}
\begin{figure}
\includegraphics[width=\textwidth]{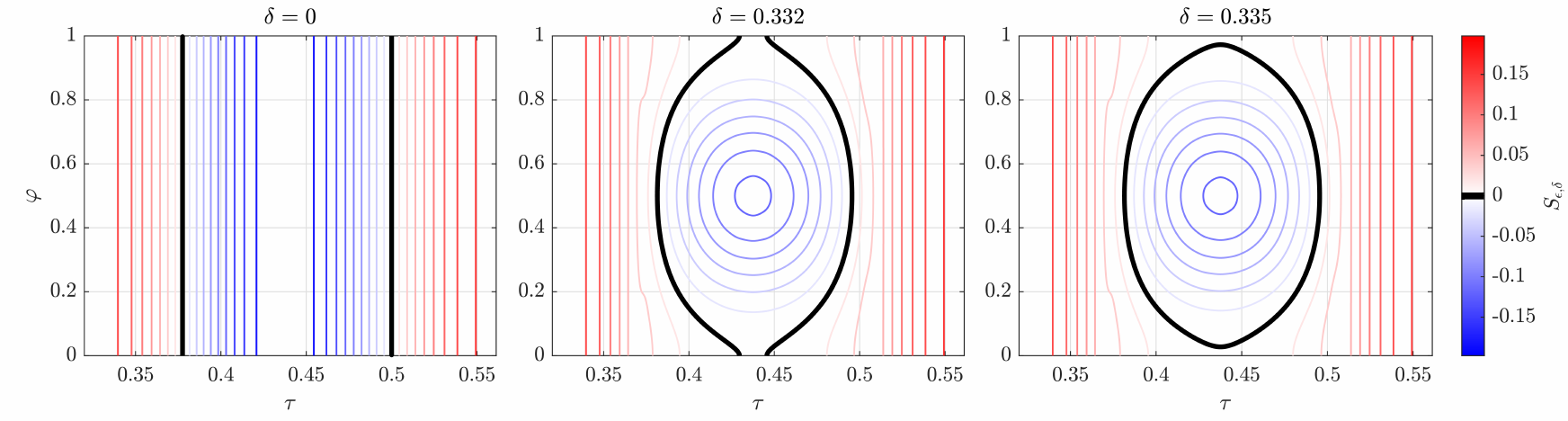}
\caption{Level curves of $S_{\ep,\delta}(\tau,\varphi)$, given in \eqref{eq:4splittingsurgical}, with the same choices in Figure~\ref{fig:4levelcurves_surgical1}, except for the values of $\delta$.
These values of $\delta$ illustrate how perturbing the parameter transforms the transverse intersection of the manifolds at $\delta=0$ into a bubble-type intersection, which subsequently evolves into the single-point intersection shown in Figure~\ref{fig:4levelcurves_surgical1}, passing through an intermediate value at which the intersection is of saddle type for some $\delta_0\in(0.332,0.335)$.
}
\label{fig:4levelcurves_surgical3}
\end{figure}

It is worth noting that this example illustrates the difference between  the two approaches for analyzing tipping induced by phase-change introduced in Section \ref{subsec:phaseinducedtipping}.
As constructed above (and illustrated in Figure \ref{fig:4levelcurves_surgical1}), there exists a unique point $(\tau_*,\varphi_*)\in[\bar\tau,A(\bar\tau))\times\T^1$ such that
$(\tau_*,\varphi_*,W^u(\tau_*,\varphi_*))\in\mW^u(\mK_-)\cap\mW^s(\mK_+)$.
By the analogue of Theorem \ref{teor:2main_result}(iii) stated in Remark \ref{rmk:2mK_+}, we obtain that

\[
\lim_{n\to\infty}d\big( F^n(\tau_*,\varphi_*,W^u(\tau_*,\varphi_*)),\mK_+\big)=0\,,
\]
while, for any $(\tau,\varphi)\in[\bar\tau,A(\bar\tau))\times\T^1\setminus\{(\tau_*,\varphi_*)\}$,
\[
\limsup_{n\to\infty} d\big(F^n(\tau,\varphi,W^u(\tau,\varphi)),\mK_+\big)>0\,.\quad 
\]
We now consider the unstable fiber attached to $(0,\varphi_0,I_1)$ for some $\varphi_0\in\T^1$ such that it does not exist $m\in\Z$ such that $\varphi_0=\varphi_*+m\,\w$, that is, the points parametrized by $\tau\mapsto(\tau,\varphi_0,W^u(\tau,\varphi_0))$.
The previous statements show that $\mK_+$ is not the $\upomega$-limit set of any such point.
This reinforces the need to choose carefully the appropriate definition of tipping due to phase, depending on the problem under analysis.
\medskip

\subsubsection{A discrete set of orbits in the intersection of 2D invariant manifolds }
One advantage of the previous construction is that we can slightly modify the definitions of the functions $q$ and $\psi$ to obtain more complicated intersections.
For example, it suffices to replace the definition of $q$ with $q(\tau):=(\tau-\tau_1)^2(\tau-\tau_2)^2$ for $\tau_1,\tau_2\in K$ to get the tangency of $\mW^u(\mK_-)$ and $\mW^s(\mK_+)$ on two orbits, those of $(\tau_1,\varphi_*,W^u_{\ep,1}(\tau_1,\varphi_*))$ and $(\tau_2,\varphi_*,W^u_{\ep,1}(\tau_2,\varphi_*))$.
We obtain the analogous effect in the other variable if we replace $\psi$ with $\psi(\varphi):=2-\cos(4\pi(\varphi-\varphi_*))$.
One such replacement is shown in Figure \ref{fig:4levelcurves_surgical2}.
Figure~\ref{fig:4huevera} replaces $\psi$ with another combination of trigonometric functions, making different intersections dissapear at different values of the parameter $\delta$.
\begin{figure}
\includegraphics[width=\textwidth]{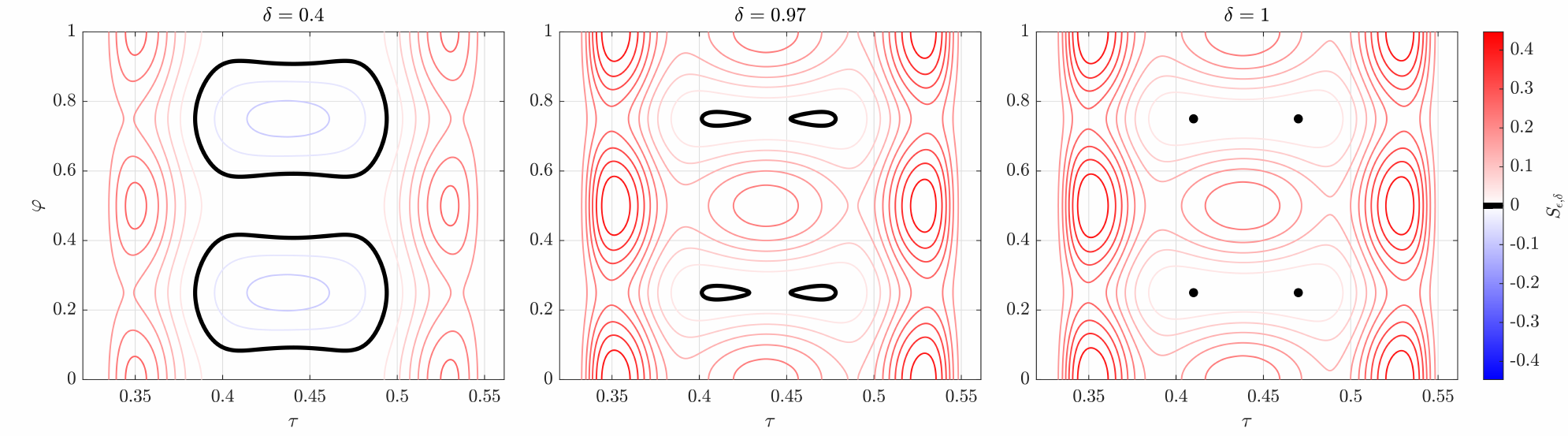}
\caption{Level curves of $S_{\ep,\delta}(\tau,\varphi)$, given in \eqref{eq:4splittingsurgical}, with the same choices of Figure \ref{fig:4levelcurves_surgical1}, except for $q(\tau)=10^{4}(\tau-0.41)^2(\tau-0.47)^2$, $U=(0.32288, 0.55817)$ and $\psi(\varphi)=2-\cos(4\pi(\varphi-0.25))$.
Four points of tangency between the manifolds are obtained.
}
\label{fig:4levelcurves_surgical2}
\end{figure}
\begin{figure}
\includegraphics[width=\textwidth]{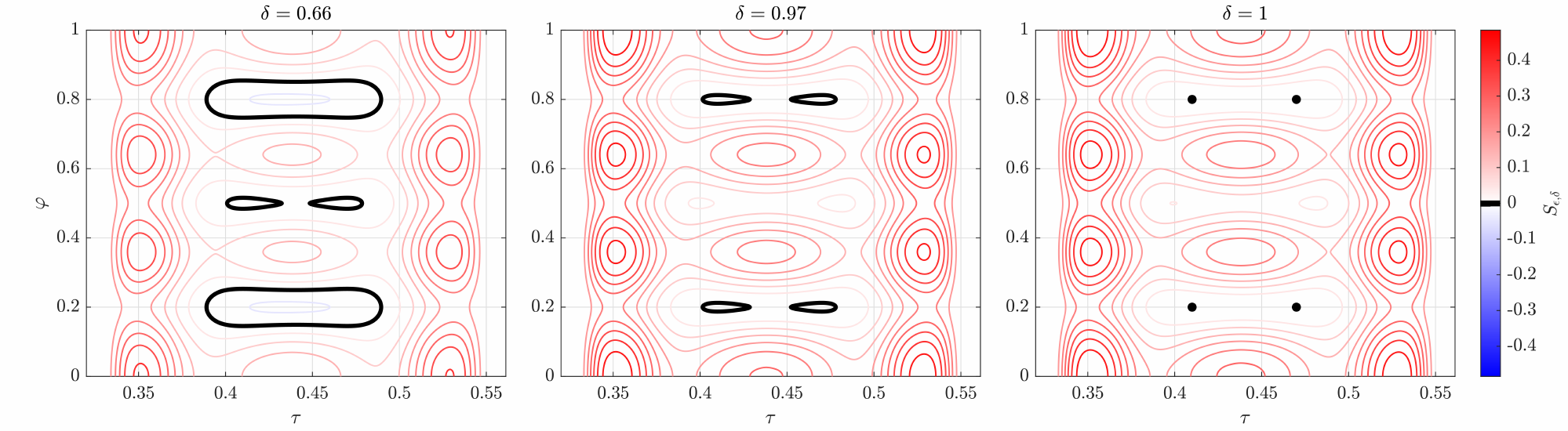}
\caption{Level curves of $S_{\ep,\delta}(\tau,\varphi)$, given in \eqref{eq:4splittingsurgical}, with the same choices of Figure \ref{fig:4levelcurves_surgical1}, except for $\psi(\varphi)=2.2470501568128455- 0.098762\cos(2\pi(\varphi-0.5))+0.525561\cos(4\pi(\varphi-0.5))-0.955835\cos(6\pi(\varphi-0.5))-0.255956\cos(8\pi(\varphi-0.5))$, which takes its global minimum (with value close to 1) at $\varphi=0.2$ and $\varphi=0.8$, and takes another local minimum at $\varphi=0.5$.
It causes the intersections near $\varphi=0.5$ to disappear before the others.
}
\label{fig:4huevera}
\end{figure}

Finally, we remark that this construction also allows us to build Cantor sets of representants of heteroclinic connections within one fundamental domain.
To this end, we can replace the function $q$ with a nonnegative $C^\infty$ function such that there exists a Cantor set $C_1\subseteq K$ such that $q(\tau)=0$ for $\tau\in C_1$ and $q(\tau)>0$ for $\tau\not\in C_1$.
Such a function can be obtained as a series of bump functions taking positive values on each gap of the complement of $C_1$.
Indeed, for a Cantor set $C$  obtained by the procedure of successively excising subintervals stages, a function $q$ as before can be  defined as follows.
Write its complement as the disjoint union $K\setminus C=\bigcup_{n\geq 1}\bigcup_{j=1}^{2^{n-1}}I_{n,j}$, where $I_{n,j}=(a_{n,j},b_{n,j})$ is one of the intervals removed at stage $n$.
In the case of a middle-third Cantor set, the interval $I_{n,j}$ has length $3^{-n}$.
On each gap, define the bump function $\phi_{n,j}(x)=\exp\bigl(-1/((x-a_{n,j})(b_{n,j}-x))\bigr)$ for $x\in I_{n,j}$, and set $\phi_{n,j}(x)=0$ otherwise. Each $\phi_{n,j}$ is nonnegative and belongs to $C^\infty$; moreover, it is strictly positive on $I_{n,j}$ and flat at both endpoints, meaning that $\phi_{n,j}^{(k)}(a_{n,j})=\phi_{n,j}^{(k)}(b_{n,j})=0$ for every $k\geq 0$.

Define $q(x)=\sum_{n\geq 1}\sum_{j=1}^{2^{n-1}}\phi_{n,j}(x)$. Since the complementary intervals are pairwise disjoint, at most one summand is nonzero at any given point. Thus, $q(x)>0$ for $x\in K\setminus C$, whereas $q(x)=0$ for $x\in C$. Notice that these bumps are not normalized to have height one. Indeed, if $I=(a,b)$ is a gap of length $l=b-a$, then $\sup_{x\in I}\phi_I(x)=\exp(-4/l^2)$. Therefore, although the supports become narrower, their heights decrease faster than any power of $l$.

More generally, for every $k\geq 0$, there exist constants $A_k,c_k>0$ and an integer $N_k\geq 0$, independent of the gap $I$, such that $\|\phi_I^{(k)}\|_\infty\leq A_k\,l^{-N_k}\exp(-c_k/l^2)$. Since $l^{-N}\exp(-c/l^2)\to 0$ as $l\to 0$, for every $N\geq 0$ and $c>0$, the derivatives of every fixed order tend uniformly to zero on the gaps as their lengths tend to zero. Consequently, each derivative of $q$ extends continuously to $C$ by assigning it the value zero there. It follows that $q\in C^\infty([0,1])$, that $q^{-1}(\{0\})=C$, and, in fact, that $q$ is flat on $C$, i.e., $q^{(k)}(x)=0$ for every $x\in C$ and every $k\geq 0$.

Analogously, we can replace $\psi$ with a function on $\T^1$ such that there exists a Cantor set $C_2\subseteq\T^1$ such that $\psi(\varphi)=1$ whenever $\varphi\in C_2$ and $\psi(\varphi)>1$ for all $\varphi\not\in C_2$.

If middle-third Cantor sets $C_1$ and $C_2$ are used, the previous construction leads to a null Lebesgue
measure Cantor set of heteroclinic connections within the fundamental domain
$[\bar\tau,A(\bar\tau)]\times\T^1$, whereas if, for example, middle-quarter
Cantor sets are used, one obtains a positive Lebesgue measure Cantor set of such heteroclinic
connections.

\section{Conclusions}

This paper contributes to the discussion of different notions of phase-induced tipping introduced in the literature and of the relationships among them in the setting of asymptotically autonomous discrete dynamical systems.
The formulation of critical transitions in terms of local pullback attractors and repellers, as developed in the existing literature, provides a natural and well-established point of reference for the present discussion.
In the compactified formulation considered here, the past and future autonomous limit systems are represented by invariant hyperplanes, which are studied as normally hyperbolic attracting and repelling invariant manifolds.
The unstable fiber bundle of the past invariant hyperplane and the stable fiber bundle of the future invariant hyperplane provide additional geometric insight into the pullback objects involved in these transitions.
This viewpoint also allows analogous connections to be considered when the relevant object of the past system is not attracting or that of the future system is not repelling, while retaining the local pullback interpretation whenever its hypotheses apply.

From this perspective, phase-induced tipping is related to the existence, absence, and geometry of intersection between the corresponding stable and unstable fiber bundles.
Such intersections identify trajectories connecting prescribed objects of the past and future systems and therefore help determine which phases may lead to a given asymptotic outcome.
Several forms of reversibility facilitate the location of
some of these intersections and the characterization---or even prevention---of phase-induced tipping. 

The splitting of stable and unstable manifolds provides a mechanism through
which the phase may determine whether a connection exists and, consequently,
whether tipping occurs, as illustrated by the examples in
Section~\ref{sec:41numericalevidenceandresults}.
As discussed in Section~\ref{sec:42surgicalintersections}, 
the first and last tangencies between invariant manifolds determine the range of
accessible $\upalpha/\upomega$-limits sets in the past/future evolution of a given
initial state of the system.
 
The framework discussed through this paper also suggests some possible directions for further study.
Note that we have considered compactified asymptotically autonomous problems in
which the invariant hyperplanes representing the past and future limits are
normally hyperbolic. While this is natural if one can consider the appropriate functions leading
to time compactification, a natural extension would be to consider
compactifications for which one or both of these invariant hyperplanes are
parabolic (that is, with polynomial decay as considered
in~\cite{wieczorek2021}), requiring constructions adapted to the resulting
nonhyperbolic dynamics.
It would also be interesting to study more complicated past and future objects, including chaotic invariant sets, for which stable and unstable fiber bundles might still provide a useful way of describing transition channels and their dependence on the forcing phase.
These questions, together with applications to specific models are left for future work.

\appendix
\section{Proof of Theorem \ref{teor:2main_result}}
\label{appendix:proofoftheorem}
Since $F$ is triangular and a diffeomorphism of $[0,1]\times\mathbb R^d$,
the map $A:[0,1]\to[0,1]$ is a $C^{\ell+1}$ diffeomorphism and
$g_0:\mathbb R^d\to\mathbb R^d$ is a $C^{\ell+1}$ diffeomorphism.

By {\rm(H1)}, after decreasing $\tau_0>0$ if necessary, we may assume that
$A'(\tau)>1$ for every $\tau\in[0,\tau_0]$.
We first turn the bunching condition (H4) on $\mU$ into a uniform condition on
a neighbourhood.
Define
\[
B(\tau,x):=
\frac{D_xG\bigl(F^{-1}(\tau,x)\bigr)}
     {A'\bigl(A^{-1}(\tau)\bigr)}
\,.
\]
For $x=g_0(y)\in\mK$, one has $B(0,x)=Dg_0(y)/A'(0)$.
Let $C^\ell_d:=C^\ell(U_0,\R^d)$ be the space of functions from $U_0$ to $\R^d$ with continuous partial derivatives up to order $\ell$, and let us endow $C^0_d$ with the supremum norm $\|h\|_{C^0_d}:=\sup_{x\in U_0}\|h(x)\|$ corresponding to the chosen norm on $\R^d$.
Then,
{\rm(H4)}, compactness of $\mU$ and continuity imply that, after
decreasing $\tau_0$ if necessary, one has
\[
\|B\|_{C^0_d}
\max\Bigl\{
1,\,
\|DF^{-1}\|_{C^0_d}^{\,\ell}
\Bigr\}
<1\,.
\]
We shall use this estimate in the sequel.

Set
\[
C(\tau,x):=
\frac{D_\tau G\bigl(F^{-1}(\tau,x)\bigr)}
     {A'\bigl(A^{-1}(\tau)\bigr)}
\,.
\]
Since $F$ is $C^{\ell+1}$, the functions $B$ and $C$ are of class $C^\ell$
on $U_0$. Consider the affine graph transform
\[
(\mT u)(\tau,x):=C(\tau,x)+B(\tau,x)\,u\bigl(F^{-1}(\tau,x)\bigr)
\]
acting on $C^\ell_d$.
This is well defined because
$F^{-1}(U_0)\subseteq U_0$.

For $\ep\in(0,1]$, endow $C^\ell_d$ with the weighted norm 
\[
\|u\|_{\ell,\ep}:=
\sum_{j=0}^{\ell}\ep^j\|D^j u\|_{C^0_d}\,.
\]
We claim that, for $\ep>0$ sufficiently small, $\mT$ is a contraction
for this norm. Indeed, if 
$u,v\in C^\ell_d$ and
$w:=u-v$, then
\[
\mT u-\mT v=B\cdot(w\circ F^{-1})\,.
\]
Let $0\le j\le\ell$. By Leibniz' rule and Faà di Bruno's formula, the
principal 
contribution to $D^j(B\cdot(w\circ F^{-1}))$ is the term in which 
all $j$ derivatives fall on $w$ rather than $B$.
This term is bounded by
\[
\|B\|_{C^0_d}
\|DF^{-1}\|_{C^0_d}^j
\|D^j w\|_{C^0_d}\,.
\]
Multiplying the principal contribution of order $j$ by $\ep^j$, and then summing over $j=0,\dots,\ell$, the sum of all principal contributions is bounded by
\[
\|B\|_{C^0_d}
\max\{1,\|DF^{-1}\|_{C^0_d}^\ell\}
\|w\|_{\ell,\ep}\,.
\]
All remaining terms contain either a derivative of $B$ or a derivative of
$w$ of order at most $j-1$. Therefore, after multiplying by $\ep^j$, each of
these lower-order terms gains an additional factor $\ep$. Consequently there
exists a polynomial $P_\ell$, depending only on $\ell$ and $d$, such that
\[
\|\mT u-\mT v\|_{\ell,\ep}
\le
\left(
\theta_0+
\ep\,P_\ell\bigl(\|B\|_{ C^\ell_d},\|F^{-1}\|_{C^\ell_d}\bigr)
\right)
\|u-v\|_{\ell,\ep}\,,
\]
where
$\theta_0:=
\|B\|_{C^0_d}
\max\{1,\|DF^{-1}\|_{C^0_d}^\ell\}<1$.

Choosing $\ep>0$ sufficiently small, $\mT$ is a contraction. By
Banach's fixed point theorem, there exists a unique $u_*\in C^\ell_d$ 
such that $\mT u_*=u_*$.

For every $(\tau,x)\in F^{-1}(U_0)$, the fixed point equation gives
\[
u_*\bigl(F(\tau,x)\bigr)
=
\frac{D_\tau G(\tau,x)+D_xG(\tau,x)\,u_*(\tau,x)}
     {A'(\tau)}
\,.
\]
Define $X(\tau,x):=(1,u_*(\tau,x))$.
Then, the previous identity is equivalent to
\[
DF(\tau,x)X(\tau,x)
=
A'(\tau)X\bigl(F(\tau,x)\bigr)
\]
for every $(\tau,x)\in F^{-1}(U_0)$. Thus the line field
$E^u(\tau,x):=\mathrm{span}(X(\tau,x))$ is $C^\ell$ and $DF$-invariant on
$F^{-1}(U_0)$. Since the first component of $X$ is identically equal to
$1$, this line field is transverse to the sections $\{\tau=\mathrm{const}\}$.

By the $C^\ell$ dependence theorem for ODEs, there exist $\delta\in(0,\tau_0]$ and a neighbourhood
$\mV_0\subseteq\mU$ with $\mV_0\subseteq\mV$ of $\mK$ such that, for every $x_0\in\mV_0$, the
initial value problem
\[
x'(\tau)=u_*(\tau,x(\tau)),\qquad x(0)=x_0\,,
\]
has a unique solution on $[0,\delta)$ not leaving $\mU$.
Denote this solution by
$v_0(\tau,x_0)$.
The map $v_0:[0,\delta)\times\mV_0\to\mathbb R^d$
is of class $C^\ell$ and satisfies $v_0(0,x_0)=x_0$.
We define
$W^u:=v_0|_{[0,\delta)\times\mK}$ and
$\Psi(\tau,x):=(\tau,W^u(\tau,x))$.
Then, {\rm(i)} follows immediately.

\medskip

We now prove the invariance identity. Fix $x\in\mK$ and $\tau\in(0,\delta)$, and note that $0<A^{-1}(\tau)<\tau$.
Define $\gamma_x(t):=\Psi(t,x)$ for $t\in[0,\delta)$ and note that the curve $\gamma_x$ is an integral curve
of $X$.
For $s\in[0,A^{-1}(\tau)]$,
consider
$\widehat\gamma_x(s):=
F^{-1}\bigl(\gamma_x(A(s))\bigr)$.
Using the invariance of $X$, one obtains
\[
\widehat\gamma_x'(s)=DF^{-1}\big(F(\widehat\gamma_x(s))\big)\,X\big(F(\widehat\gamma_x(s))\big)\,A'(s)=X(\widehat\gamma_x(s))\,.
\]
Moreover,
$\widehat\gamma_x(0)=F^{-1}(0,x)=(0,g_0^{-1}(x))$. By uniqueness of
integral curves of $X$ it follows that
\[
F^{-1}\bigl(\Psi(\tau,x)\bigr)
=
\Psi\bigl(A^{-1}(\tau),g_0^{-1}(x)\bigr)\,.
\]
Since $g_0(\mK)=\mK$, this gives
{\rm(ii)} for every $x\in\mK$.

\medskip

We next show that $\Psi:[0,\delta)\times\mK\to[0,\delta)\times\mathbb R^d$ is a $C^\ell$ embedding.
Since its first component is $\tau$, 
points with  different  $\tau$ have different images. For a fixed
$\tau$, the map $x\mapsto v_0(\tau,x)$ is the time-$\tau$ map of the
nonautonomous ODE $x'=u_*(s,x)$, and is therefore a $C^\ell$
diffeomorphism onto its image. 
In particular, $\Psi$ is an injective immersion and defines a  $C^\ell$ embedding.
This proves {\rm(iv)}.

\medskip

Let
$
\mW^u_{\mathrm{loc}}(\mK_-):=
\bigl\{
(\tau,W^u(\tau,x))\,\big|\,\tau\in(0,\delta),\ x\in\mK
\bigr\}\,.
$
%\]
By the invariance identity~{\rm(ii)}, if
$z=\Psi(\tau,x)\in\mW^u_{\mathrm{loc}}(\mK_-)$, then
$F^{-n}(z)=\Psi\bigl(A^{-n}(\tau),g_0^{-n}(x)\bigr)$
for all $n\ge0$.
Since $A^{-n}(\tau)\to0$ and $g_0^{-n}(x)\in\mK$, we get
$\dist\bigl(F^{-n}(z),\{0\}\times\mK\bigr)\to0$.
This proves the inclusion ``$\subseteq$'' in {\rm(iii)}.

Conversely, let $z\in(0,\delta)\times\mathbb R^d$ satisfy
$\dist\bigl(F^{-n}(z),\{0\}\times\mK\bigr)\to0$.
Write $z_n:=F^{-n}(z)$. Recall that the map
\[
\Phi:[0,\delta)\times\mV_0\to[0,\delta)\times\mathbb R^d,
\qquad
\Phi(\tau,y):=(\tau,v_0(\tau,y))\,,
\]
is also a $C^\ell$ diffeomorphism onto its image. Moreover, $\Phi(0,y)=(0,y)$.
Thus $\Phi([0,\delta)\times\mV_0)$ is a neighbourhood of
$\{0\}\times\mK$ in $[0,\delta)\times\mathbb R^d$. Hence, there exists $N\in\N$ such that $z_n\in\Phi([0,\delta)\times\mV_0)$ for all $n\geq N$.
Therefore, for every $n\ge N$,
there exist $\tau_n\in[0,\delta)$ and $y_n\in\mV_0$ such that
$z_n=\Phi(\tau_n,y_n)$.
The invariance identity for the extended parametrization $\Phi$ gives
$F^{-1}\bigl(\Phi(\tau,y)\bigr)
=
\Phi\bigl(A^{-1}(\tau),g_0^{-1}(y)\bigr)
$.
Since the first component of $F$ is $A$, if $\tau_z$ denotes the first
component of $z$, then necessarily $\tau_n=A^{-n}(\tau_z)$.
Hence, for every $n\ge N$,
$y_{n+1}=g_0^{-1}(y_n)$.
Consequently, $y_{N+j}=g_0^{-j}(y_N)$ for every $j\ge0$.

Since $\tau_n\to0$, $\Phi(0,y)=(0,y)$, and
$\dist(z_n,\{0\}\times\mK)\to0$, we get
$\dist(y_n,\mK)\to0$.
Therefore, the $\upalpha$-limit set of $y_N$ is contained in $\mK$.
By {\rm(H3)},
this implies $y_N\in\mK$.
Consequently,
$z_N=\Phi(\tau_N,y_N)=\Psi(\tau_N,y_N)\in\mW^u_{\mathrm{loc}}(\mK_-)$.
Using the forward invariance identity finitely many times, we conclude that
$z=F^N(z_N)\in\mW^u_{\mathrm{loc}}(\mK_-)$.
This proves the reverse
inclusion in {\rm(iii)}.

Finally, uniqueness follows from the dynamical characterization {\rm(iii)}.
Indeed, any other $C^\ell$ map with the same characterization has the same
image $\mW^u_{\mathrm{loc}}(\mK_-)$. Since every point of this image has a
unique first coordinate $\tau$ and the sections are graphs over $\mK$,
the corresponding fiber map must coincide with $W^u$.

\medskip

It remains to prove the last assertion under the additional assumption that
$(\mK,g_0)$ is conjugate to a rigid rotation.
Let
$h:\mK\to\mathbb T^m$ be a homeomorphism such that
$h\circ g_0|_{\mK}=R_\omega\circ h$, where $R_\omega$ is an irrational rotation of
$\mathbb T^m$.

First, note that $W^u(\tau,x)$ satisfies the required asymptotic property.
Indeed, by the invariance identity in {\rm(ii)}, for every
$\tau\in(0,\delta)$ and every $x\in\mK$ we have
\[
F^{-n}\bigl(\tau,W^u(\tau,x)\bigr)
=
\bigl(A^{-n}(\tau),W^u(A^{-n}(\tau),g_0^{-n}(x))\bigr)\,.
\]
Since $A^{-n}(\tau)\to0$ and $W^u(0,z)=z$ for every $z\in\mK$, it follows
that
\[
\dist\Bigl(
F^{-n}\bigl(\tau,W^u(\tau,x)\bigr),
\bigl(0,g_0^{-n}(x)\bigr)
\Bigr)\to0\,.
\]

We now prove uniqueness. Let $y\in\mathbb R^d$ satisfy the property.
Then, in particular,
\[
\dist\bigl(F^{-n}(\tau,y),\{0\}\times\mK\bigr)\to0\,,
\]
and, by {\rm(iii)}, one has $(\tau,y)\in\mW^u_{\mathrm{loc}}(\mK_-)$. Hence there exists
$z\in\mK$ such that $y=W^u(\tau,z)$. Using again the invariance identity, we get
\[
F^{-n}(\tau,y)
=
F^{-n}\bigl(\tau,W^u(\tau,z)\bigr)
=
\bigl(A^{-n}(\tau),W^u(A^{-n}(\tau),g_0^{-n}(z))\bigr)\,.
\]
Since $A^{-n}(\tau)\to0$, the assumed convergence to
$(0,g_0^{-n}(x))$ ensures that
\[
\dist\bigl(W^u(A^{-n}(\tau),g_0^{-n}(z)),g_0^{-n}(x)\bigr)\to0\,.
\]
Since the sequence $(g_0^{-n}(z))_n$ is bounded in $\mK$,
$\dist\bigl(W^u(A^{-n}(\tau),g_0^{-n}(z)),g_0^{-n}(z)\bigr)\to0$ and therefore
$\dist\bigl(g_0^{-n}(z),g_0^{-n}(x)\bigr)\to0$.
Applying the conjugacy $h$, we obtain
$\dist_{\mathbb T^m}\bigl(
R_\omega^{-n}h(z),R_\omega^{-n}h(x)
\bigr)\to0$.
But $R_\omega$ is an isometry of $\mathbb T^m$, and therefore
\[
\dist_{\mathbb T^m}\bigl(
R_\omega^{-n}h(z),R_\omega^{-n}h(x)
\bigr)
=
\dist_{\mathbb T^m}\bigl(h(z),h(x)\bigr)
\]
for every $n\ge0$. Hence $\dist_{\mathbb T^m}(h(z),h(x))=0$, so
$h(z)=h(x)$. Since $h$ is injective, $z=x$. Consequently
$y=W^u(\tau,x)$.
This proves the uniqueness claimed in the last assertion.

\begin{thebibliography}{99}
\bibitem{AlexanderGardnerJones1990} J.~Alexander, R.A.~Gardner, C.K.R.T.~Jones. A topological invariant arising in the stability analysis of travelling waves, \emph{J. Reine Angew. Math.} \textbf{410}, 167--212 (1990).
\bibitem{alk2018} H.~Alkhayuon, P.~Ashwin. Rate-induced tipping from periodic attractors: Partial tipping and connecting orbits. \textit{Chaos} \textbf{28}, 033608 (2018).
\bibitem{alk2023} H.~Alkhayuon, J.~Marley, S.~Wieczorek, R.C.~Tyson. Stochastic resonance in climate reddening increases the risk of cyclic ecosystem extinction via phase-tipping. \emph{Global Change Biology} \textbf{29}, 3347–3363 (2023).
\bibitem{alk2021} H.~Alkhayuon, R.C.~Tyson, S.~Wieczorek. Phase tipping: how cyclic ecosystems respond to contemporary climate. \emph{Proc. Roy. Soc. A} \textbf{477} (2254), 20210059 (2021).
\bibitem{AnagnostopoulouPotzscheRasmussen2023} V.~Anagnostopoulou, C.~P{\"o}tzsche, M.~Rasmussen. \emph{Nonautonomous Bifurcation Theory: Concepts and Tools}. Springer, 2023.
%\bibitem{ashwin1994}
%P.~Ashwin, J.~Buescu, I.~Stewart. Bubbling of attractors and synchronisation of chaotic oscillators. \emph{Phys. Lett. A} \textbf{193}, 126--139 (1994).
\bibitem{apw2017} P.~Ashwin, C.~Perryman, S.~Wieczorek. Parameter shifts for nonautonomous systems in low dimension: Bifurcation- and Rate-induced tipping. \textit{Nonlinearity} \textbf{30} (6), 2185 (2017).
\bibitem{AshwinWieczorekVitoloCox2012} P.~Ashwin, S.~Wieczorek, R.~Vitolo, P.~Cox. Tipping points in open systems: Bifurcation, noise-induced and rate-dependent examples in the climate system. \emph{Philos. Trans. R. Soc. A.} \textbf{370} (1962), 1166--1184 (2012).
\bibitem{canadell2015} M.~Canadell, R.~de la Llave. KAM tori and whiskered invariant tori for non-autonomous systems. \textit{Physica D} \textbf{310}, 104–113 (2015).
%\bibitem{cartwright2002}
%J.H.E.~Cartwright, M.O.~Magnasco, O.~Piro. Bailout embeddings, targeting of invariant tori, and the control of Hamiltonian chaos. \emph{Phys. Rev. E} \textbf{65}, 045203(R) (2002).
\bibitem{cima2015} A.~Cima, A.~Gasull, F.~Mañosas, R.~Ortega.
Linearization of planar involutions in $C^1$.
\textit{Annali di Matematica} \textbf{194}, 1349–1357 (2015). 
\bibitem{lno1} J.~Due\~{n}as, I.P.~Longo, R.~Obaya. Rate-induced tracking for concave or d-concave transitions in a time-dependent environment with application in ecology. \textit{Chaos} \textbf{33} (12), 123113 (2023).
\bibitem{dno3} J.~Due\~{n}as, C.~N\'{u}\~{n}ez, R.~Obaya. Critical transitions in d-concave nonautonomous scalar ordinary differential equations appearing in population dynamics. \textit{SIAM J. Appl. Dyn. Syst.} \textbf{22} (4), 2649-2692 (2023).
\bibitem{dno4} J.~Due\~{n}as, C.~N\'{u}\~{n}ez, R.~Obaya. Critical transitions for asymptotically concave or d-concave nonautonomous differential equations with applications in ecology. \textit{J. Nonlinear Sci.} \textbf{34}, 105 (2024).
\bibitem{eldering2013}
J.~Eldering.
\textit{Normally Hyperbolic Invariant Manifolds: The Noncompact Case.} Atlantis Series in Dynamical Systems \textbf{2}.
Springer, 2013.
\bibitem{fenichel1974} N.~Fenichel.
Asymptotic stability with rate conditions for dynamical systems.
\textit{Bull. Amer. Math. Soc.} \textbf{80} (2), 346-349 (1974).
\bibitem{fonsimvie19} E.~Fontich, C.~Sim{\'{o}}, A.~Vieiro, Splitting of the separatrices after a Hamiltonian{\textendash}Hopf bifurcation under periodic forcing. \textit{Nonlinearity} \textbf{4}, 1440--1493 (2019).
\bibitem{mamotreto} \`{A}.~Haro, M.~Canadell, J.L.~Figueras, A.~Luque, J.M.~Mondelo. \textit{The Parameterization Method for Invariant Manifolds.} Applied Mathematical Sciences \textbf{195}, Springer-Verlag, Switzerland, 2016.
\bibitem{harodelallave2006}
\`{A}.~Haro, R.~de la Llave.
A parameterization method for the computation
of invariant tori and their whiskers
in quasi-periodic maps: Rigorous results.
\textit{J. Differential Equations} \textbf{228}, 530–579 (2006).
\bibitem{hirsch1977}
M.W.~Hirsch, C.C.~Pugh, M.~Shub.
\textit{Invariant Manifolds.}
Springer, 1977.
\bibitem{hoyerleitzel2021}
A.~Hoyer-Leitzel, A.N.~Nadeau.
Rethinking the definition of rate-induced tipping.
\emph{Chaos} \textbf{31}, 053133 (2021).
\bibitem{jorba2001} \`{A}.~Jorba. Numerical computation of the normal behaviour of
invariant curves of n-dimensional maps. \textit{Nonlinearity} \textbf{14}, 943 (2001).
\bibitem{kiers2020} C.~Kiers.
Rate-induced tipping in discrete-time dynamical systems.
\emph{SIAM J. Appl. Dyn. Syst.} \textbf{19} (2), 1200-1224 (2020).
\bibitem{KloedenPotzsche2013}
P.E.~Kloeden, C.~Pötzsche (eds). \emph{Nonautonomous Dynamical Systems in the Life Sciences}. Lecture Notes in Mathematics \textbf{2102}. Springer, Cham, 2013.
\bibitem{lamb1998}
J.S.W.~Lamb, J.A.G.~Roberts.
Time-reversal symmetry in dynamical systems: A survey.
\emph{Phys. D} \textbf{112}, 1-39 (1998).
\bibitem{lamb2026}
J.S.W.~Lamb, M.~Rasmussen, W.H.~Tey. Bifurcations of the Hénon Map with additive bounded noise. \textit{SIAM J. Appl. Dyn. Syst.} \textbf{25} (1), 351-374 (2026).
\bibitem{LomeliRamirez2008}
H.E.~Lomel\'i, R.~Ram\'irez-Ros. Separatrix splitting in 3D volume-preserving maps.  \textit{SIAM J. Appl. Dyn. Syst.} \textbf{7} (4), 1527–1557 (2008).
\bibitem{LongoNunezObaya2021}
I.P.~Longo, C.~Núñez, R.~Obaya. Critical transitions in piecewise uniformly continuous concave quadratic ordinary differential equations. \emph{J. Dynam. Differential Equations} \textbf{36}, 2153--2192 (2024).
\bibitem{longonunezobaya}
I.P.~Longo, C.~Núñez, R.~Obaya.
Critical transitions for scalar nonautonomous systems with concave nonlinearities: some rigorous estimates.
\textit{Nonlinearity} \textbf{37} (4), 045017 (2024).
\bibitem{lnor}
I.P.~Longo, C.~Núñez, R.~Obaya, M.~Rasmussen.
Rate-induced tipping and saddle-node bifurcation for quadratic differential equations with nonautonomous asymptotic dynamics.
\textit{SIAM J. Appl. Dyn. Syst.} \textbf{20} (1), 500--540 (2021).
\bibitem{Markus1956}
L.~Markus. Asymptotically autonomous differential systems, in S.~Lefschetz (ed.). \emph{Contributions to the Theory of Nonlinear Oscillations, Vol.~III}, Annals of Mathematics Studies \textbf{36}, Princeton University Press, 1956.
\bibitem{murillo} A.~Murillo, A.~Vieiro. Periodic perturbation of a 3D conservative flow with a heteroclinic connection to saddle-foci. \textit{Commun. Nonlinear Sci. Numer. Simul.} \textbf{143}, 108602 (2025).
%\bibitem{neishtadt1987}
%A.I.~Neishtadt. On the change in the adiabatic invariant on crossing a separatrix in systems with two degrees of freedom. \emph{J. Appl. Math. Mech.} \textbf{51}, 586--592 (1987).
%\bibitem{neishtadttreschev2021}
%A.I. Neishtadt, D.V. Treschev. Dynamical phenomena connected with stability loss of equilibria and periodic trajectories. \emph{Russian Math. Surveys} \textbf{76} (5), 883--926 (2021).
%\bibitem{neishtadtvasiliev2005} A.I.~Neishtadt and A.A.~Vasiliev. Phase change between separatrix crossings in slow-fast Hamiltonian systems. \emph{Nonlinearity} \textbf{18}, 1393--1406 (2005).
\bibitem{nippstoffer} K.~Nipp and D.~Stoffer,
\emph{Invariant Manifolds in Discrete and Continuous Dynamical Systems}.
EMS Tracts in Mathematics \textbf{21}, European Mathematical Society, Z\"urich, 2013.
\bibitem{pesin} Y.B.~Pesin. \textit{Lectures on partial hyperbolicity and stable ergodicity}. Zurich Lectures in Advanced Mathematics, European Mathematical Society, Zürich, 2004.
\bibitem{rasmussen2007} M.~Rasmussen. \textit{Attractivity and Bifurcation for Nonautonomous Dynamical Systems}. Lecture Notes in Mathematics \textbf{1907}, Springer-Verlag, Berlin, 2007.
\bibitem{rasmussen2008} M.~Rasmussen. Bifurcations of Asymptotically Autonomous
Differential Equations. \textit{Set-Valued Anal.} \textbf{16}, 821–849 (2008).
\bibitem{scarcella2024} D.~Scarcella. Asymptotically quasiperiodic solutions
for time-dependent Hamiltonians. \textit{Nonlinearity} \textbf{37}, 065005 (2024).
\bibitem{scarcella2024-2} D.~Scarcella. Biasymptotically quasi-periodic solutions for time-dependent Hamiltonians. \textit{Regul. Chaotic Dyn.} \textbf{29}, 620-653 (2024).
%\bibitem{scarcella2025} D.~Scarcella. Weakly asymptotically quasiperiodic solutions for time-dependent Hamiltonians with a view to celestial mechanics. \textit{J. Differential Equations} \textbf{431}, 113192 (2025).
\bibitem{vieirophd} A.~Vieiro.
Study of the effect of conservative and weakly dissipative perturbations on symplectic maps and Hamiltonian systems.
\emph{PhD thesis}, Universitat de Barcelona, 2009.
\bibitem{wieczorek2023}
S.~Wieczorek, C.~Xie, P.~Ashwin.
Rate-induced tipping: thresholds, edge states and connecting orbits.
\emph{Nonlinearity} \textbf{36}, 3238--3293 (2023).
\bibitem{wieczorek2021} S.~Wieczorek, C.~Xie, C.K.R.T.~Jones. Compactification for asymptotically autonomous dynamical systems: theory, applications and invariant manifolds. \textit{Nonlinearity} \textbf{34}, 2970 (2021).
\bibitem{wiggins} S.~Wiggins.
\textit{Normally Hyperbolic Invariant Manifolds in Dynamical System.}
Applied Mathematical Sciences \textbf{105}, Springer-Verlag, New York, 1994.
\end{thebibliography}
\end{document}